\documentclass[11pt]{amsart}
\usepackage{amsmath,amsthm,amsfonts,amssymb,graphicx,psfrag,dsfont}
\usepackage{color}

\usepackage{hyperref}
\usepackage[alphabetic,initials]{amsrefs}

\theoremstyle{plain}
\newtheorem{prop}{Proposition}[section]
\newtheorem{thm}[prop]{Theorem}
\newtheorem{conj}[prop]{Conjecture}
\newtheorem{cor}[prop]{Corollary}

\newtheorem{lemma}[prop]{Lemma}

\theoremstyle{definition}
\newtheorem{example}[prop]{Example}

\newtheorem{defn}[prop]{Definition}

\theoremstyle{remark}
\newtheorem{remark}[prop]{Remark}

\newcommand{\Aut}{\operatorname{Aut}}

\newcommand{\coker}{\operatorname{coker}}
\newcommand{\conjug}{^\textbf{c}}

\newcommand{\defin}{\textbf}

\newcommand{\Diff}{\operatorname{Diff}}
\newcommand{\dist}{\operatorname{dist}}
\newcommand{\End}{\operatorname{End}}
\newcommand{\ev}{\operatorname{ev}}

\newcommand{\Fred}{\operatorname{Fred}}
\newcommand{\GW}{\operatorname{GW}}

\newcommand{\Hom}{\operatorname{Hom}}

\newcommand{\im}{\operatorname{im}}
\newcommand{\ind}{\operatorname{ind}}

\newcommand{\Jfix}{J_{\operatorname{fix}}}

\newcommand{\PD}{\operatorname{PD}}

\renewcommand{\Re}{\operatorname{Re}}
\newcommand{\reg}{{\operatorname{reg}}}

\newcommand{\transpose}{{\operatorname{T}}}

\newcommand{\vol}{d\operatorname{vol}}

\newcommand{\GL}{\operatorname{GL}}

\newcommand{\Ortho}{\operatorname{O}}

\newcommand{\U}{\operatorname{U}}

\newcommand{\SO}{\operatorname{SO}}

\newcommand{\CC}{{\mathbb C}}

\newcommand{\NN}{{\mathbb N}}
\newcommand{\QQ}{{\mathbb Q}}
\newcommand{\RR}{{\mathbb R}}
\newcommand{\TT}{{\mathbb T}}
\newcommand{\ZZ}{{\mathbb Z}}

\newcommand{\jJ}{{\mathcal J}}
\newcommand{\lL}{{\mathcal L}}
\newcommand{\mM}{{\mathcal M}}

\newcommand{\oO}{{\mathcal O}}

\newcommand{\tT}{{\mathcal T}}
\newcommand{\uU}{{\mathcal U}}

\newcommand{\1}{\mathds{1}}
\newcommand{\p}{\partial}
\renewcommand{\dbar}{\bar{\partial}}

\newcommand{\tame}{{^{\operatorname{tame}}}}
\newcommand{\comp}{{^{\operatorname{comp}}}}

\numberwithin{equation}{section}

\definecolor{blue}{rgb}{0,0,1}
\definecolor{red}{rgb}{1,0,0}
\definecolor{green}{rgb}{0,.7,0}

\newcommand{\rev}[1]{#1}

\title[Generic transversality for unbranched covers of holomorphic curves]{Generic transversality for unbranched covers of closed pseudoholomorphic curves}

\author{Chris Gerig}
\address{Mathematics Department\\
970 Evans Hall\\
University of California\\
Berkeley CA 94720\\
USA}
\email{cgerig@berkeley.edu}

\author{Chris Wendl}
\address{Institut f\"ur Mathematik \\ 
Humboldt-Universit\"at zu Berlin \\
Unter den Linden 6 \\
10099 Berlin \\ 
Germany}
\email{wendl@math.hu-berlin.de}

\thanks{CG is partially supported by NSF grants DMS-0838703 and DMS-1105820.
CW was partially supported during this work by a Royal Society University Research Fellowship and a 
Leverhulme Research Project Grant.}
\subjclass[2010]{Primary 32Q65; Secondary 57R17, 53D45}

\begin{document}

\begin{abstract}
We prove that in closed almost complex manifolds of any dimension, 
generic perturbations of
the almost complex structure suffice to achieve transversality
for all unbranched multiple covers of simple pseudoholomorphic curves 
with deformation index zero.
A corollary is that the Gromov-Witten invariants (without descendants) of
symplectic $4$-manifolds can always be computed as a signed and weighted
count of honest $J$-holomorphic curves for generic tame~$J$: in particular,
each such invariant is an integer divided by a weighting factor that depends only
on the divisibility of the corresponding homology class.
The transversality proof is based on an analytic 
perturbation technique, originally due to Taubes.
\end{abstract}

\maketitle

\tableofcontents

\section{Introduction}
\label{sec:intro}

The Gromov-Witten invariants of closed symplectic manifolds are defined in
principle by counting $J$-holomorphic curves for generic tame almost complex
structures~$J$.  One of the main technical hurdles in this
definition is that moduli spaces of $J$-holomorphic curves are not generally
manifolds of the ``expected'' dimension unless multiply covered curves can
be excluded; thus in practice, the definition usually requires more 
sophisticated
techniques such as virtual cycles, abstract multivalued perturbations, or
stabilizing divisors, see e.g.~\cites{FukayaOno,LiTian,Ruan:virtual,
Siebert:GW,CieliebakMohnke:transversality,IonelParker:virtual,HWZ:GW}.

It is nonetheless interesting to ask under what circumstances
the ``classical'' technique
of perturbing $J$ generically suffices for a complete description of
moduli spaces of multiply covered curves.  Results of this nature are
desirable for several reasons: one is that the resulting definition of
the Gromov-Witten invariants is simpler to understand and to apply.
Another is that the relationship between simple curves and their
multiple covers can reveal nontrivial relations among Gromov-Witten
invariants that cannot be seen by more abstract techniques; 
one example of this phenomenon is the Gopakumar-Vafa conjecture on
symplectic Calabi-Yau $3$-folds,
see \cites{GopakumarVafa,BryanPandharipande:BPS,BryanPandharipande:local,
IonelParker:GV}.
While moduli spaces of multiply covered curves cannot
generally achieve regularity in the usual sense, it is sometimes enough
to show that they are \emph{as regular as possible}.  A simple
$J$-holomorphic curve $u$
with deformation index~$0$ is called ``super-rigid'' if, roughly
speaking, the set of all covers of $u$ is an open subset in the
moduli space of all $J$-holomorphic curves (see \S\ref{sec:main} for a more
precise definition), so in particular, no sequence
of curves geometrically distinct from $u$ can converge to any cover of~$u$.
The index relations between simple $J$-holomorphic
curves and their multiple covers make the following conjecture
plausible:\footnote{After this article was submitted for publication,
the second author produced a preprint \cite{Wendl:super} that proves Conjecture~\ref{conj:super}
in all dimensions greater than four, together with a substantial generalization of
Theorem~\ref{thm:super}, using different techniques based on the
Sard-Smale theorem and representation theory.}

\begin{conj}
\label{conj:super}
On any closed symplectic manifold $(M,\omega)$ of real dimension at least
four, there exists a Baire
subset $\jJ_\reg$ in the space of smooth $\omega$-tame almost complex
structures such that for all $J \in \jJ_\reg$, every closed, connected
and simple $J$-holomorphic curve with deformation index~$0$ is
super-rigid.
\end{conj}

Some special cases of this conjecture have been proved previously by
Lee-Parker \cites{LeeParker:structure,LeeParker:obstruction}
and Eftekhary \cite{Eftekhary:superrigidity}.  The techniques used in
the present paper are related to those of
\cites{LeeParker:structure,LeeParker:obstruction}, which also play a role
in the announced solution by Ionel and Parker to the Gopakumar-Vafa conjecture
\cite{IonelParker:GV}.

For an \emph{unbranched} cover of a simple curve, the super-rigidity
condition is equivalent to the usual notion of \emph{Fredholm regularity},  
and our main result (stated as Theorem~\ref{thm:super} below) is that
this can always be achieved by choosing $J$ generically.  This may
be seen as an initial step toward a proof of
Conjecture~\ref{conj:super} in full generality.
While the result holds in all dimensions, its consequences are especially
interesting in dimension four: as we will show in \S\ref{sec:GW}, it implies
that Gromov-Witten invariants without descendants 
in this setting can be computed without the
aid of domain-dependent or inhomogeneous perturbations, and they therefore
satisfy integrality conditions that are not apparent from the more
general definitions; see Theorem~\ref{thm:finitely} and
Corollary~\ref{cor:integral}.

Our proof is quite different from the methods that symplectic topologists 
typically use to establish transversality: it does not involve the
Sard-Smale theorem, but is instead based on an analytic perturbation theory 
technique introduced by Taubes in his definition of the Gromov 
invariants of symplectic $4$-manifolds \cite{Taubes:SWtoGr}.  
It works in the symplectic
category in all dimensions greater than two, but it does not work 
in the algebraic or complex category, i.e.~if we start with an integrable
complex structure~$J$, then our perturbation to achieve regularity will
\emph{always} make $J$ nonintegrable (see Remark~\ref{remark:peculiarities}).
The method also is not strictly limited to unbranched covers: for any given
cover of a simple curve with index~$0$, we will show how to
perturb $J$ such that the super-rigidity condition is achieved for the given 
cover.  Since spaces of unbranched covers do not have moduli, this suffices to
prove our main result, and it also lends hope that similar methods could
be used to prove Conjecture~\ref{conj:super} in full generality, though at
present it is not clear whether the kind of perturbation we define can
achieve super-rigidity for all branched covers at once in a space with
nontrivial moduli.\footnote{A preliminary version of this paper (under a
different title) claimed a proof of Conjecture~\ref{conj:super}
using similar techniques, but this argument had gaps that we have thus far
been unable to fill.  See Remark~\ref{remark:fail}.}

We aim in future work to prove similar results for covers of
finite-energy punctured $J$-holomorphic curves in
symplectic cobordisms, which should have interesting applications in
Symplectic Field Theory \cite{SFT} and 
Embedded Contact Homology \cite{Hutchings:lectures}.
A few special cases of super-rigidity in the punctured case have
previously been observed by the second author \cite{Wendl:automatic}, as
well as work of Fabert \cite{Fabert:local}, and unpublished work of Hutchings
\cite{Hutchings:magic}; those examples were restricted to
dimension four, but the methods introduced in the present paper
have no such restrictions.

\subsection{The main result}
\label{sec:main}

Assume $(M,\Jfix)$ is an almost complex manifold of dimension $2n \ge 4$,
$\uU \subset M$ is an open subset with compact closure, and
$$
\jJ(M\,;\,\uU,\Jfix)
$$
denotes the space of smooth almost complex structures on $M$ that
match $\Jfix$ outside of~$\uU$, with its natural $C^\infty$-topology.
If $M$ also carries a symplectic structure $\omega$ for which
$\Jfix$ is $\omega$-tame or $\omega$-compatible, we will denote the
corresponding spaces of tame/compatible almost complex structures
matching $\Jfix$ outside~$\uU$ by
$$
\jJ\tame(M,\omega\,;\,\uU,\Jfix) , \ 
\jJ\comp(M,\omega\,;\,\uU,\Jfix) \subset \jJ(M\,;\,\uU,\Jfix).
$$
\begin{remark}
\label{remark:symp}
The existence of a symplectic form on $M$ is not required for
any of the arguments in this paper, but since it is important
in applications, we will generally
assume at least that $(M,\omega)$ is symplectic and all
almost complex structures under consideration are
$\omega$-tame.  Note that
$\jJ\tame(M,\omega\,;\,\uU,\Jfix)$ is an open subset of
$\jJ(M\,;\,\uU,\Jfix)$, thus all statements made about
$\jJ\tame(M,\omega\,;\,\uU,\Jfix)$ will have obvious analogues
for $\jJ(M\,;\,\uU,\Jfix)$.
\end{remark}

With Remark~\ref{remark:symp} in mind, from now on we fix a
symplectic form $\omega$ on $M$ and assume $\Jfix$ is
$\omega$-tame.
Given $J \in \jJ\tame(M,\omega\,;\,\uU,\Jfix)$, a closed connected Riemann surface
$(\Sigma,j)$ and a $J$-holomorphic curve\footnote{When we use the word
``curve'' to describe $u : (\Sigma,j) \to (M,J)$, we mean that
$(\Sigma,j)$ is a smooth (non-nodal) Riemann surface and $u$ is a smooth map,
or in some cases an equivalence class of smooth maps up to parametrization
(this will be clear from context).  By
default this excludes nodal curves, and when we do mean ``nodal curve''
we will make this explicit.  This usage is common in
symplectic topology but may differ from conventions in the algebraic geometry
literature.}
$u : (\Sigma,j) \to (M,J)$, the \defin{index} of $u$ is the integer
\begin{equation}
\label{eqn:index}
\ind(u) = (n-3) \chi(\Sigma) + 2 c_1(u),
\end{equation}
where we abbreviate $c_1(u) := \langle c_1(TM,J), [u] \rangle$,
$[u] := u_*[\Sigma] \in H_2(M)$.  A closed and connected $J$-holomorphic curve
$\tilde{u} : (\widetilde{\Sigma},\tilde{\jmath}) \to (M,J)$ is said to be
a ($d$-fold) \defin{multiple cover} of $u$ if $\tilde{u} = u \circ \varphi$
for some holomorphic map $\varphi : (\widetilde{\Sigma},\tilde{\jmath}) \to
(\Sigma,j)$ of degree $d \ge 2$, and $u$ is called \defin{simple} if it is
nonconstant and is not a multiple cover of any other curve.  The map
$\varphi : \widetilde{\Sigma} \to \Sigma$ is generally a branched cover,
and we call it \defin{unbranched} (and $\tilde{u}$ an \emph{unbranched cover
of $u$}) if it is an honest covering map, meaning its set of branch points
is empty.

We say that the curve $u : \Sigma \to M$ is \defin{Fredholm regular} if a
neighborhood of $u$ in the moduli space of unparametrized $J$-holomorphic
curves is cut out transversely, 
see e.g.~\cite{Wendl:lecturesV2}*{\S 4.3}.  In this paper we will mainly
deal with immersed curves, for which a precise
definition of regularity is easier to state: suppose $u : \Sigma \to M$
is immersed and denote its complex normal bundle by $N_u \to \Sigma$.
The linearized Cauchy-Riemann operator associated to $u$ is the 
real-linear first-order differential operator
\begin{equation}
\label{eqn:Du}
\mathbf{D}_u : \Gamma(u^*TM) \to \Omega^{0,1}(\Sigma,u^*TM) :
\eta \mapsto \nabla \eta + J(u) \circ \nabla\eta \circ j +
(\nabla_\eta J) \circ Tu \circ j,
\end{equation}
where $\nabla$ is any choice of symmetric connection on~$M$.
We define
the \defin{normal Cauchy-Riemann operator} at $u$ as the restriction of
$\mathbf{D}_u$ to sections of $N_u$, composed with
the projection $\pi_N : u^*TM \to N_u$, hence
$$
\mathbf{D}_u^N = \pi_N \circ \mathbf{D}_u|_{\Gamma(N_u)} : \Gamma(N_u) \to 
\Omega^{0,1}(\Sigma,N_u).
$$
This is also a Cauchy-Riemann type operator, so its extension to any
reasonable Banach space completions such as
\begin{equation}
\label{eqn:normalBanach}
\mathbf{D}_u^N : W^{k,p}(N_u) \to W^{k-1,p}(\overline{\Hom}_\CC(T\Sigma,N_u))
\end{equation}
for $k \in \NN$ and $p > 1$ is a Fredholm operator, 
and elliptic regularity implies that its
kernel and cokernel do not depend on the choices $k$ and~$p$.
The curve $u$ is then Fredholm regular if and only if the linear map
\eqref{eqn:normalBanach} is surjective.  
In the present paper, we will sometimes deal with multiple covers
$\tilde{u} = u \circ \varphi$ for which $u$ is immersed but $\varphi$ may have
branch points, in which case $\mathbf{D}_{\tilde{u}}^N$ can naturally be 
defined as a Cauchy-Riemann type operator on 
$N_{\tilde{u}} := \varphi^*N_u$.  The curve $u$ is then called 
\defin{super-rigid} if it is immersed with index~$0$ and
$\mathbf{D}_{\tilde{u}}^N$ is injective for every cover $\tilde{u}$
of~$u$.  Note that if $\varphi : \widetilde{\Sigma} \to \Sigma$ has 
degree $d \in \NN$ and $Z(d\varphi) \ge 0$ denotes the
number of branch points of $\varphi$ counted with multiplicities,
then the Riemann-Hurwitz formula 
\begin{equation}
\label{eqn:RiemannHurwitz}
-\chi(\widetilde{\Sigma}) + d\chi(\Sigma)
= Z(d\varphi)
\end{equation}
implies
$$
\ind(\tilde{u}) = d \cdot \ind(u) - (n-3) Z(d\varphi),
$$
hence unbranched covers of immersed index~$0$ curves are also immersed
with index~$0$, and super-rigidity for unbranched covers is therefore
the same as Fredholm regularity.

Here is our main result.

\begin{thm}
\label{thm:super}
Assume $(M,\omega)$ is a symplectic manifold\footnote{As indicated in
Remark~\ref{remark:symp}, the first statement in the theorem could
also be stated without reference to any symplectic structure,
producing a Baire subset of $\jJ(M\,;\,\uU,\Jfix)$.}
with tame almost complex structure
$\Jfix$, and $\uU$ is an open subset with compact closure.  Then
there exists a Baire subset $\jJ_\reg \subset \jJ\tame(M,\omega\,;\,\uU,\Jfix)$
such that for every $J \in \jJ_\reg$, all unbranched covers of simple closed
$J$-holomorphic curves of index~$0$ contained fully in~$\uU$ are Fredholm
regular.

Moreover, if $\Jfix$ is $\omega$-compatible, then there is a Baire subset
$\jJ_\reg \subset \jJ\comp(M,\omega\,;\,\uU,\Jfix)$ such that for every
$J \in \jJ_\reg$, all unbranched covers of \emph{embedded} closed 
$J$-holomorphic curves of index~$0$ contained fully in~$\uU$ are
Fredholm regular.
\end{thm}

\begin{remark}
\label{remark:weDoNotKnow}
We do not know whether the restriction to embedded curves in the
$\omega$-compatible case can be relaxed; the reason is
explained in Remark~\ref{remark:whyNotCompatible}.
This is in any case only
a restriction in dimension four, since embeddedness is a generic property
of holomorphic curves in higher dimensions
(see e.g.~\cite{Wendl:lecturesV2}*{\S 4.6} or \cite{OhZhu:embedding}).
In the $\omega$-tame case, our argument works for all immersed curves
with distinct transverse self-intersections, which is a generic
property even in dimension four.
\end{remark}

The next two remarks draw attention to generalizations of
Theorem~\ref{thm:super} that might naturally be expected to hold but
do \emph{not} follow from our arguments,
and in some cases are actually false.

\begin{remark}
\label{remark:deformation}
The standard transversality results as in \cites{McDuffSalamon:Jhol,Wendl:lecturesV2}
for simple $J$-ho\-lo\-morph\-ic curves 
have straightforward extensions to generic $1$-paramater families
$\{ J_\tau \}$ of almost complex structures, showing in essence that the
space of pairs 
$$
\{ (\tau,u) \ |\ \text{$u$ is {simple and} $J_\tau$-holomorphic} \}
$$
is a manifold of dimension $\ind(u) + 1$.  This means that all simple
$J_\tau$-holomorphic curves are regular for almost every~$\tau$, but
there may be birth-death bifurcations at a discrete set of parameter
values.  The work of Taubes \cite{Taubes:counting} shows that
when multiple covers are {allowed}, more general types of
bifurcations must be {considered}, so e.g.~the extension of the usual
results for simple curves to unbranched covers of index~$0$ curves
is not at all straightforward.  We will not prove anything in this paper
about generic $1$-parameter families of data.
\end{remark}
\begin{remark}
\label{remark:local}
The standard results for simple curves do not require the curves to be \emph{fully}
contained in the perturbation domain~$\uU$ in order to achieve transversality; 
it suffices rather that they should intersect
$\uU$ \emph{somewhere}, the key point being that there is an injective point
mapped into~$\uU$.  Our methods on the other hand work only for curves that
are fully contained in $\uU$, and we do not know whether this assumption
can be weakened.  The reason for this is discussed in 
Remark~\ref{remark:peculiarities}.
In this sense, Theorem~\ref{thm:super} seems to represent
a fundamentally different phenomenon 
from the usual transversality results for simple curves.
\end{remark}

\subsection{Application to Gromov-Witten theory}
\label{sec:GW}

In the results of this section,
the words ``for generic $J$\ldots'' should be understood to mean
that there exists a Baire subset of the appropriate space of almost
complex structures for which the statement is true.  

Let $\mM_{g,m}(A,J)$ denote the moduli space of smooth unparametrized
$J$-holomorphic curves in $M$ with genus~$g$ and $m$ marked points in the
homology class $A \in H_2(M)$; the precise definition will be recalled
in the discussion below.  We denote the natural evaluation map by
$$
\ev : \mM_{g,m}(A,J) \to M^m,
$$
and let 
$$
\mM_{g,m}^*(A,J) \subset \mM_{g,m}(A,J)
$$
denote the open subset consisting of simple curves.  For any
integer $m \ge 0$, the $m$-point Gromov-Witten invariant
$$
\GW^{(M,\omega)}_{g,m,A} : H^*(M)^{\otimes m} \to \QQ
$$
is defined morally by counting intersections of the evaluation
map with cycles in $M^m$ determined by an $m$-tuple of cohomology
classes.  The standard definition of these invariants in 
\cite{RuanTian:higherGenus} for semipositive
symplectic manifolds (which includes all symplectic $4$-manifolds)
requires generic inhomogeneous perturbations to the nonlinear
Cauchy-Riemann equation, thus breaking the symmetry inherent in
multiply covered curves.  We will now show that when $\dim_\RR M = 4$,
these invariants can also be computed by simpler means that do not
break the symmetry.
Recall from \cite{McDuffSalamon:Jhol}*{\S 6.5}
that for any subset $\mM^* \subset \mM_{g,m}(A,J)$,
the restriction $\ev : \mM^* \to M^m$ is said to be a
\defin{pseudocycle of dimension $d \ge 0$} if $\mM^*$ is a smooth $d$-dimensional
manifold and $\overline{\mM}_{g,m}(A,J) \setminus \mM^*$ can be covered by 
subsets on which
$\ev$ factors through a smooth map to $M^m$ from a manifold of dimension
at most $d-2$.  In this case one can define integer-valued 
intersection products of $\ev$ with homology classes in~$M^m$.
The following proposition for the case $m \ge 1$ is presumably not
a new result, but we are not aware of any proof of it in the current
literature; ours
will require only the standard transversality results for simple curves.

\begin{prop}
\label{prop:pseudo}
Assume $(M,\omega)$ is a closed symplectic $4$-manifold.  Then for
generic $\omega$-compatible or tame almost complex structures $J$ and for
every $A \in H_2(M)$ and every pair of nonnegative integers $(g,m)$
satisfying $-(2-2g) + 2c_1(A) > 0$ and $m \ge 1$,
the evaluation map $\ev : \mM_{g,m}^*(A,J) \to M^m$ on the set of
simple curves 
is a pseudocycle of dimension $-(2-2g) + 2 c_1(A) + 2m$.  The
corresponding $m$-point Gromov-Witten invariant can thus be computed as
an intersection number
$$
\GW^{(M,\omega)}_{g,m,A}(\alpha_1,\ldots,\alpha_m) = \left[\ev|_{\mM_{g,m}^*(A,J)}\right]
\cdot \left( \PD(\alpha_1) \times \ldots \times \PD(\alpha_m) \right),
$$
and in particular, its values are always integers.
\end{prop}

The picture for the $0$-point invariants with $g \ge 1$ is somewhat different, 
as it turns out that multiply covered curves cannot be avoided in this case,
but only \emph{unbranched} covers need be considered.  The arguments
behind Proposition~\ref{prop:pseudo} thus combine with Theorem~\ref{thm:super}
to give the following more novel result.

\begin{thm}
\label{thm:finitely}
For generic $\omega$-tame almost complex structures $J$ on a closed
symplectic $4$-manifold $(M,\omega)$, the set of index~$0$ curves
satisfying any given bound on their genus and area is finite, 
and all of them are Fredholm regular.
\end{thm}

We should again caution the reader that we do not know whether the
generic $J$ in Theorem~\ref{thm:finitely} can be chosen to be
\emph{compatible} with~$\omega$ (see Remark~\ref{remark:weDoNotKnow}), 
though one can require this if one
is only interested in covers of embedded curves 
(as in \cites{Taubes:counting,Taubes:SWtoGr}).  Choosing $J$ tame is in any case
good enough to compute Gromov-Witten invariants.
In order to state the main corollary, we can associate to any integral
homology class $A \in H_2(M)$ in a symplectic manifold $(M,\omega)$ its
\defin{symplectic divisibility}
$$
d_\omega(A) \in \NN,
$$
defined as the product of the finite set of integers $k \in \NN$ such that 
$A = kB$ for some primitive class $B \in H_2(M)$ with $\omega(B) > 0$.

\begin{cor}
\label{cor:integral}
Suppose $(M,\omega)$ is a closed symplectic $4$-manifold and
$A \in H_2(M)$ and $g \in \NN$ satisfy $-(2-2g) + 2c_1(A) = 0$.  Then the
$0$-point Gromov-Witten invariant can be computed
for generic tame almost complex structures $J$ as a signed and weighted
count of finitely many $J$-holomorphic curves
$$
\GW^{(M,\omega)}_{g,0,A} = \sum_{u \in \mM_{g,0}(A,J)} \frac{\sigma(u)}{|\Aut(u)|},
$$
where for each curve $u$, $\sigma(u) \in \{-1,1\}$ is determined by an orientation
of the determinant line bundle, and $\Aut(u)$ denotes the
automorphism group of~$u$.  In particular, the number
$\GW^{(M,\omega)}_{0,0,A}$ is always an integer, while for $g \ge 1$,
$d_\omega(A) \cdot \GW^{(M,\omega)}_{g,0,A}$ is an integer.
\end{cor}

In order to prepare for the proofs of these results,
let us recall the definitions of the relevant
moduli spaces.  Given
integers $g,m \ge 0$ and a homology class $A \in H_2(M)$, 
the moduli space of \defin{unparametrized $J$-holomorphic curves}
$\mM_{g,m}(A,J)$ can be defined as the set
of equivalence classes of tuples
$(\Sigma,j,\Theta,u)$ where $(\Sigma,j)$ is a closed connected Riemann
surface of genus~$g$, $\Theta \subset \Sigma$ is an ordered set of
$m$ distinct points (the \defin{marked points}), 
and $u : (\Sigma,j) \to (M,J)$ is a $J$-holomorphic
map satisfying $[u] = A$, with equivalence defined by $(\Sigma,j,\Theta,u) \sim
(\Sigma',\psi^*j,\psi^{-1}(\Theta),u \circ \psi)$ for diffeomorphisms
$\psi : \Sigma' \to \Sigma$.  The \defin{automorphism group} 
$\Aut(u)$ of $[(\Sigma,j,\Theta,u)] \in \mM_{g,m}(A,J)$ is the group of
biholomorphic diffeomorphisms $\psi : (\Sigma,j) \to (\Sigma,j)$ that fix
each of the marked points and satisfy $u = u \circ \psi$; it is always finite,
and is trivial whenever $u$ is simple.
The \defin{Gromov compactification} of
$\mM_{g,m}(A,J)$ is the space $\overline{\mM}_{g,m}(A,J)$ of 
(equivalence classes of) \defin{stable nodal
curves} $(S,j,\Theta,\Delta,u)$, where now $S$ may be disconnected, and
the original data are augmented by an unordered
set of distinct points in $S \setminus \Theta$, arranged into unordered pairs
$$
\Delta = \left\{ \{\hat{z}_1,\check{z}_1\},\ldots,\{\hat{z}_r,\check{z}_r\} \right\},
$$
such that $u(\hat{z}_i) = u(\check{z}_i)$ for each $i=1,\ldots,r$.  
We call the pairs $\{\hat{z}_i,\check{z}_i\}$ \defin{nodes}, and each individual
$\hat{z}_i$ or $\check{z}_i \in S$ a \defin{nodal point}.
The curves in
$\overline{\mM}_{g,m}(A,J)$ are required to have \defin{arithmetic genus}~$g$,
which means that the surface obtained from $S$ by performing connected sums
at all matched pairs of nodal points is a closed connected
surface of genus~$g$.  The stability condition requires that any component of
$S \setminus (\Theta \cup \Delta)$ on which $u$ is constant should have negative
Euler characteristic.  With this condition, $\overline{\mM}_{g,m}(A,J)$ can be
given a natural topology as a metrizable Hausdorff space, and it is compact
whenever $J$ is tamed by a symplectic form.  A definition of the topology may be
found e.g.~in \cite{SFTcompactness}; for
sequences in $\mM_{g,m}(A,J)$, it amounts to the notion of $C^\infty$-convergence
for $j$ and $u$ after a choice of parametrization for which all domains and
marked point sets are identified.  Curves $[(S,j,\Theta,\Delta,u)] \in
\overline{\mM}_{g,m}(A,J)$ with $\Delta = \emptyset$ can equivalently be
regarded as elements of $\mM_{g,m}(A,J)$, and are thus called \defin{smooth}
curves to distinguish them from nodal curves.
The evaluation map is defined by
$$
\ev : \mM_{g,m}(A,J) \to M \times \ldots \times M :
[(\Sigma,j,(\zeta_1,\ldots,\zeta_m),u)] \mapsto (u(\zeta_1),\ldots,u(\zeta_m)),
$$
and it extends to a continuous map on $\overline{\mM}_{g,m}(A,J)$.

When there is no danger of confusion, we shall sometimes 
abuse notation by denoting equivalence classes $[(\Sigma,j,\Theta,u)] \in \mM_{g,m}(A,J)$
or $[({S},j,\Theta,\Delta,u)] \in \overline{\mM}_{g,m}(A,J)$ simply by 
$u \in \mM_{g,m}(A,J)$ or $u \in \overline{\mM}_{g,m}(A,J)$ respectively, 
and we will refer to the
restriction of a nodal curve {$[(S,j,\Theta,\Delta,u)]$ to any connected component of 
its domain $S$} as a
\defin{smooth component} of~$u$.  Recall that $\mM_{g,0}(A,J)$ has
\defin{virtual dimension} equal to the index of any curve $u \in \mM_{g,0}(A,J)$.

It will be useful to recall certain index relations for
degenerating sequences of holomorphic curves.
Suppose $\dim_\RR M = 2n$, and $[(\Sigma,j_k,u_k)] 
\in \mM_{g,0}(A,J)$ is a sequence 
converging to a stable nodal curve $[(S,j_\infty,\Delta,u_\infty)] 
\in \overline{\mM}_{g,0}(A,J)$ with {smooth} components
$$
\left\{ [(S_i,j_\infty^i,u_\infty^i)] \in \mM_{g_i}(A_i,J) \right\}_{i=1,\ldots,r}.
$$
Then if $N_i := | S_i \cap \Delta | \ge 1$ denotes the number of nodal points on $S_i$ for 
$i=1,\ldots,r$, we have $\chi(\Sigma) = \sum_i \left[ \chi(S_i) - N_i \right]$,
so the index formula \eqref{eqn:index} gives
\begin{equation}
\label{eqn:indexAdd}
\ind(u_k) = \sum_{i=1}^r \left[ \ind(u_\infty^i) - (n-3) N_i \right].
\end{equation}
Note that by the stability condition, we have 
\begin{equation}
\label{eqn:stability}
\chi(S_i) - N_i < 0 \quad \text{ whenever $A_i = 0$}.
\end{equation}
If $A_i \ne 0$, then $u_\infty^i = v^i \circ \varphi^i$ for some
simple curve $v^i$ and holomorphic map $\varphi^i$ of degree
$d_i \ge 1$ with $Z(d\varphi^i) \ge 0$ branch points, and the Riemann-Hurwitz
formula combined with \eqref{eqn:index} gives
\begin{equation}
\label{eqn:cover}
\ind(u_\infty^i) = d_i \cdot \ind(v^i) - (n-3) Z(d\varphi^i).
\end{equation}

\begin{proof}[Proof of Proposition~\ref{prop:pseudo}] 
Assume $J$ is chosen so that all somewhere injective curves are
Fredholm regular.  Then $\mM_{g,m}^*(A,J)$ is a manifold of real 
dimension $\ind(u) + 2m$ for any $u \in \mM_{g,m}^*(A,J)$.
The index relations \eqref{eqn:indexAdd} and \eqref{eqn:cover} 
imply that if $u_k \in \mM_{g,m}^*(A,J)$ is a 
sequence of simple curves with $\ind(u_k) > 0$ converging to a 
nodal curve $u_\infty$, then the nonconstant
components of $u_\infty$ cover simple curves whose indices add up to
at most $\ind(u_k) - 2$.  More concretely, if
$u_\infty$ has {smooth} components $u_\infty^1,\ldots,u_\infty^r$,
each $u_\infty^i$ having $N_i \ge 1$ nodal points,
then the $4$-dimensional case of \eqref{eqn:indexAdd} together with
the stability condition \eqref{eqn:stability} implies
\begin{equation}
\label{eqn:indexNodal}
\ind(u_k) \ge \sum_{\{ i\ |\ u_\infty^i \ne \text{const} \}} \left[ 
\ind(u_\infty^i) + N_i \right],
\end{equation}
with equality if and only if $u_\infty$ has no
constant (i.e.~``ghost'') components.  This shows in particular that
\begin{equation}
\label{eqn:indexNodal2}
\ind(u_k) \ge 2 + \sum_{\{ i\ |\ u_\infty^i \ne \text{const} \}} \ind(u_\infty^i).
\end{equation}
Now by \eqref{eqn:cover} in the case $n=2$, we see that
if $u_\infty^i$ is a $d_i$-fold cover of a simple curve $v^i$, then
$\ind(u_\infty^i) \ge d_i \ind(v^i)$, with equality if and only if
the cover is unbranched.  Since $\ind(v^i) \ge 0$ by genericity,
this implies that each smooth component $u^i_\infty$ has index at
least two less than $\ind(u_k)$.  On the other hand, if
$u_\infty = \lim u_k$ is a smooth curve that is a $d$-fold cover 
$v \circ \varphi$ of some simple curve $v$, then \eqref{eqn:cover} gives
$$
\ind(u_\infty) = d \cdot \ind(v) + Z(d\varphi) \ge d \cdot \ind(v),
$$
and since $\ind(u_\infty) > 0$ by assumption and the index is always
even, we conclude
$\ind(v) \le \ind(u_\infty) - 2$ unless $d=1$.
These relations imply the pseudocycle condition.
\end{proof}

\begin{proof}[Proof of Theorem~\ref{thm:finitely} and Corollary~\ref{cor:integral}]
Applying the index relations as in the proof of Proposition~\ref{prop:pseudo}
above, we find that the worst case scenario for a degenerating
sequence of index~$0$ curves $u_k \to u_\infty$ is that $u_\infty$ is
an \emph{unbranched} cover of a simple index~$0$ curve.  For generic
tame~$J$, Theorem~\ref{thm:super} implies that the latter is regular,
hence all curves in $\overline{\mM}_{g,0}(A,J)$ are smooth and regular, 
and therefore isolated due to the implicit function theorem.
The integrality condition in Corollary~\ref{cor:integral} arises from the
observation that whenever $u \in \mM_{g,0}(A,J)$ is a $d$-fold cover of
a simple curve $v \in \mM_{g',0}(B,J)$, we necessarily 
have $A = dB$ and $\omega(B) > 0$, and the order of the automorphism group
$\Aut(u)$ is an integer dividing~$d$.  For $g=0$ the integrality result is
stronger, because the Riemann-Hurwitz formula forbids the existence of
unbranched covers with genus~$0$, hence every curve in
$\mM_{0,0}(A,J)$ is simple.
\end{proof}

\subsection{Outline of the paper}

The main steps in the proof of Theorem~\ref{thm:super}
will be explained in \S\ref{sec:proof}, modulo three technical
results concerning (1)~the nonlinear problem, (2)~the linear problem,
and (3)~obstruction theory.  The remainder of the
paper will then be concerned with these three technical results:
the nonlinear result in \S\ref{sec:nonlinear},
the linear result in \rev{\S\ref{sec:Weitzenbock} and }\S\ref{sec:linear}, and the 
obstruction theoretic
result (which is only needed for the case $\dim_\RR M \ge 6$) in
\S\ref{sec:symmetry}.  These are followed by a brief appendix
recalling the essential result from analytic perturbation
theory that is needed in \S\ref{sec:linear}.

\subsection*{A brief remark on terminology}

{Since many important objects in this paper do not carry natural
complex structures, our formulas for dimensions and Fredholm indices
generally give the \emph{real} dimension unless otherwise noted,
even in cases where this number is always even.  The major exceptions
are the bundles $u^*TM$ and $N_u$ associated to a $J$-holomorphic
curve $u : (\Sigma,j) \to (M,J)$; these are naturally complex
vector bundles and are described in terms of their \emph{complex} rank.}

\subsection*{Acknowledgements}

The present paper emerged out of discussions between the two authors and
Michael Hutchings and Dan Cristofaro-Gardiner at the Simons Center's
\textsl{Workshop on Moduli Spaces of Pseudo-holomorphic Curves~II},
June 2--6, 2014.
We would like to thank Hutchings and Cristofaro-Gardiner 
for contributing useful ideas and encouragement, Helmut Hofer,
Dusa McDuff, Tim Pertuz, Cliff Taubes and 
Aleksey Zinger
for enlightening conversations, Daniel Rauch for sending us a copy of
his PhD thesis, and the Simons Center for its hospitality and for
providing such a stimulating environment for collaboration.
We also thank Eleny Ionel and Tom Parker for pointing out a crucial
error in our preliminary version of this paper.

\section{\rev{The main argument}}
\label{sec:proof}

The goal of this section will be to reduce the proof of 
Theorem~\ref{thm:super} to a sequence of three technical results to be proved 
in later sections.

\subsection{Unbranched tori in dimension four}
\label{sec:Taubes}

Before diving into the details on Theorem~\ref{thm:super}, 
it may be instructive to recall the argument 
of Taubes which has inspired the present approach to regularity for
multiple covers.
The Gromov invariants were defined in \cite{Taubes:counting,Taubes:SWtoGr} as certain
counts of holomorphic curves in symplectic $4$-manifolds, including both
embedded curves and unbranched covers of embedded holomorphic tori with
index~$0$.
In order to achieve transversality for the multiple covers, Taubes argued
in \cite{Taubes:SWtoGr}*{\S 7(b)}
as follows.  Assume $u : \TT^2 \to M$ is an embedded $J$-holomorphic torus with
index~$0$, $\varphi : \TT^2 \to \TT^2$ is a holomorphic covering map and
$\tilde{u} = u \circ \varphi$.  Then the normal Cauchy-Riemann operator
for $\tilde{u}$ can be identified with an operator of the form
$$
\mathbf{D} = \dbar + A : C^\infty(\TT^2,\CC) \to C^\infty(\TT^2,\CC),
$$
where $\dbar = \p_s + i\p_t$ in holomorphic coordinates $s + it$ on $\TT^2$
and $A \in C^\infty(\TT^2,\End_\RR(\CC))$.  Taubes shows that one can always
perturb the ambient almost complex structure along $u$ such that $\mathbf{D}$ becomes
$$
\mathbf{D}_\tau \eta := \mathbf{D}\eta + \tau \beta \bar{\eta}
$$
for some $\beta \in C^\infty(\TT^2,\CC^*)$ and a small parameter
$\tau \in \RR$.  This perturbation of the
linear operator is required to be complex-antilinear, and it must never vanish, 
but in contrast to the standard transversality arguments as in 
\cite{McDuffSalamon:Jhol}, it is allowed to be arbitrarily symmetric, so in
particular the fact that $\tilde{u}$ is a multiple cover poses no difficulty
here.  The main challenge is now to show that this perturbed
operator will always be injective for sufficiently small $\tau > 0$.
The argument for this involves two main ingredients.

\textsl{(1) Bochner-Weitzenb\"ock technique}:
The following argument shows that $\mathbf{D}_\tau$ must be injective
for all $\tau \gg 0$.  \rev{Fix the standard real-valued $L^2$-inner product
on $C^\infty(\TT^2,\CC)$ and let $\mathbf{D}^*$ and $\mathbf{D}_\tau^*$
denote the formal adjoints of $\mathbf{D}$ and $\mathbf{D}_\tau$ respectively;
explicitly, we have $\mathbf{D}^* = \p + A^*$ and 
$\mathbf{D}^*_\tau \eta = \mathbf{D}^* \eta + \tau \beta \bar{\eta}$,
where $\p = \p_s - i \p_t$ and $A^* \in C^\infty(\TT^2,\End_\RR(\CC))$ 
denotes the pointwise real-linear transpose of~$A$.  From these relations,
one obtains a Weitzenb\"ock formula,
\begin{equation}
\label{eqn:WeitzenbockDD}
\mathbf{D}^*_\tau \mathbf{D}_\tau \eta = \mathbf{D}^* \mathbf{D} \eta +
\tau L \eta + \tau^2 |\beta|^2 \eta,
\end{equation}
where $L \in C^\infty(\TT^2,\End_\RR(\CC))$ is the zeroth-order real-linear
operator $L\eta = \beta \overline{A \eta} + A^* \beta \bar{\eta} -
(\p \beta) \bar{\eta}$.  The crucial point in \eqref{eqn:WeitzenbockDD}
is that $\mathbf{D}^*_\tau \mathbf{D}_\tau \eta$ and $\mathbf{D}^*\mathbf{D}\eta$
differ only by a zeroth-order term---the
complex-\emph{anti}linear nature of the perturbation causes all other
derivatives of $\eta$ to cancel.
For all $\eta \in C^\infty(\TT^2,\CC)$, we then have
\begin{equation}
\begin{split}
\label{eqn:Weitzenbock}
\| \mathbf{D}_\tau \eta \|_{L^2}^2 &= 
\langle \eta , \mathbf{D}^*_\tau \mathbf{D}_\tau \eta \rangle_{L^2}
= \left\langle \eta , \mathbf{D}^*\mathbf{D} \eta + \tau L\eta + 
\tau^2 |\beta|^2 \eta \right\rangle_{L^2} \\
&= \| \mathbf{D} \eta \|_{L^2}^2 + \tau \langle \eta, L\eta \rangle_{L^2} +
\tau^2 \langle \eta , |\beta|^2 \eta \rangle_{L^2} \\ 
&\ge \| \mathbf{D} \eta \|_{L^2}^2 + (c \tau^2 - c' \tau) \| \eta \|_{L^2}^2
\end{split}
\end{equation}
for some constants $c , c' > 0$.  Here we have used the fact that
$\beta$ is nowhere zero so that $\langle \eta,|\beta|^2 \eta \rangle_{L^2} \ge
c \| \eta \|_{L^2}^2$.
}

\textsl{(2) Analytic perturbation theory}:
Regard $\mathbf{D}_\tau$ as a complex-linear operator 
$H^1(\TT^2,\CC) \to L^2(\TT^2,\CC)$, or more accurately on the 
complexifications of these two spaces.  Then $\mathbf{D}_\tau$ depends
analytically on the parameter $\tau \in \CC$, so the set of all
$\tau \in \CC$ for which $\mathbf{D}_\tau$ is not an isomorphism looks
locally like the zero-set of an analytic function on~$\CC$, 
i.e.~$\mathbf{D}_\tau$ has nontrivial kernel either for all~$\tau$ or only
for a discrete subset.  (A proof of this fact is given in the Appendix.)  
Step~(1) implies that it is the latter, not the former.

\begin{remark}
\label{remark:peculiarities}
The first step described above depends crucially on the following two
properties of the perturbation, both of which lend a distinctive
flavor to our main result:
\begin{enumerate}
\item The perturbation \rev{from $\mathbf{D}$ to $\mathbf{D}_\tau$} 
must be \emph{antilinear}, \rev{otherwise the Weitzenb\"ock formula
\eqref{eqn:WeitzenbockDD} does not hold}.
This implies that, in general, the generic almost
complex structures for which our transversality result holds can
\emph{never} be expected to be integrable.
\item The perturbation must also be \emph{nowhere zero} so that
$\| \eta \|_{L^2}$ can be bounded below via
\rev{$\langle \eta |\beta|^2 \eta \rangle_{L^2}$}
in \eqref{eqn:Weitzenbock}.  This is why our proof of Theorem~\ref{thm:super}
does not work for curves that only pass through the
perturbation domain rather than being fully contained in it
(see Remark~\ref{remark:local}).
\end{enumerate}
We will see that both of these features also appear in the general case
to be discussed below.
\end{remark}

\begin{remark}
A version of the Bochner-Weitzenb\"ock technique described above has also appeared
in the work of Lee and Parker on K\"ahler surfaces with positive geometric
genus, see 
\cite{LeeParker:structure}*{Proposition~8.6}.  In their more specialized
setting, the terms linear in $\tau$ vanish for geometric reasons,
thus one obtains super-rigidity for all (not necessarily small)
perturbations of the type that they consider, without any
need to apply analytic perturbation theory.
\end{remark}

\subsection{\rev{Three technical results for the general case}}

We now describe what is required in order to generalize the argument
of Taubes sketched above.

The first technical result we will need describes the perturbation of the 
normal Cauchy-Riemann operator realized by a certain class of
perturbations to the almost complex structure.  Working under the
assumptions of Theorem~\ref{thm:super}, suppose
$u : (\Sigma,j) \to (M,J)$ is an immersed
$J$-holomorphic curve with image fully contained in~$\uU$, 
choose a tangent/normal splitting $u^*TM = T_u \oplus N_u$
with $T_u = \im du$, and abbreviate 
the complex vector bundles
$$
E := N_u, \qquad F := \overline{\Hom}_\CC(T\Sigma,N_u) = T^{0,1}\Sigma \otimes E,
$$
both of which have rank $m := n-1$.  The normal Cauchy-Riemann operator
$\mathbf{D}_u^N$ then maps sections of $E$ to sections of~$F$.
Suppose $\{ J_\tau \in \jJ\tame(M,\omega\,;\,\uU,\Jfix) \}_{\tau \in (-\epsilon,\epsilon)}$
is a smooth $1$-parameter family of almost complex structures such that 
$$
J_0 \equiv J, \quad\text{ and }\quad
J_\tau|_{T_u} \equiv J|_{T_u} \text{ for all $\tau$}.
$$
Then $u : (\Sigma,j) \to (M,J_\tau)$ 
is $J_\tau$-holomorphic for all~$\tau$, though the previously chosen
normal bundle $N_u \subset u^*TM$ may fail to be $J_\tau$-invariant
for $\tau \ne 0$.  Nonetheless one can always find a smooth $1$-parameter
family of complex bundle isomorphisms
$$
\Phi_\tau : (TM,J) \to (TM,J_\tau)
$$
that fix $T_u$ and satisfy $\Phi_0 = \1$, 
allowing us to define perturbed complex normal bundles 
$N_{u,\tau} := \Phi_\tau(N_u)$ and
normal Cauchy-Riemann operators
$$
\mathbf{D}_{u,\tau}^N : \Gamma(N_{u,\tau}) \to \Gamma(\overline{\Hom}_\CC(T\Sigma,N_{u,\tau})),
$$
so that a $1$-parameter family of operators
$\Gamma(E) \to \Gamma(F)$ can be defined by
$$
\Phi_\tau^{-1}  \mathbf{D}_{u,\tau}^N \Phi_\tau : \Gamma(E) \to \Gamma(F).
$$
We will prove the following result in \S\ref{sec:nonlinear}.

\begin{prop}
\label{prop:nonlinear}
Assume the curve $u : (\Sigma,j) \to (M,J)$ in the above setup is
immersed with only transverse double points, such that no point in $M$
is in the image of more than two distinct points of~$\Sigma$.
Then given any real-linear bundle map $B : E \to F$, one can choose the 
families of $\omega$-tame almost complex structures
$\{ J_\tau \}$ and complex bundle isomorphisms $\{ \Phi_\tau \}$ as above 
such that 
$$
\Phi_\tau^{-1} \mathbf{D}_{u,\tau}^N \Phi_\tau = \mathbf{D}_u^N + \tau B.
$$
In particular, for any $p > 1$, this defines a family of Fredholm operators
$W^{1,p}(E) \to L^p(F)$ that depends analytically on
the parameter~$\tau$.  If $J$ is $\omega$-compatible and $u$ has
no double points, then one can also arrange that
$J_\tau \in \jJ\comp(M,\omega\,;\,\uU,\Jfix)$ for all~$\tau$.
\end{prop}

Continuing with the above setup, assume now that
$\ind(u) = 0$.  Then $0$ is also the index of $\mathbf{D}_u^N$, which
is $m \chi(\Sigma) + 2 c_1(E)$, hence $-c_1(E) = m \chi(\Sigma) + c_1(E) =
c_1(F)$, implying the existence of a complex-antilinear bundle isomorphism
$B : E \to F$.  Let $\langle \ ,\ \rangle$ denote a Hermitian bundle metric
on $E$, and denote its real part by $\langle\ ,\ \rangle_\RR$; if
$J$ is $\omega$-compatible, we may assume that
$\langle\ ,\ \rangle_\RR$ matches
the restriction of $\omega(\cdot,J\cdot)$ to $N_u$.
For our linear transversality argument, it will be
important to establish the following symmetry property for $B$, which
will be possible due to an obstruction theoretic argument explained
in~\S\ref{sec:symmetry}.
{Note that the condition described here is vacuous when $E$ is a
line bundle, so this step did not appear in Taubes's argument of
\S\ref{sec:Taubes} and is only needed for the higher-dimensional case.}

\begin{prop}
\label{prop:symmetry}
Every homotopy class of complex-antilinear bundle isomorphisms 
$B : E \to \overline{\Hom}_\CC(T\Sigma,E)$ contains one that satisfies the
following condition: for all
$z \in \Sigma$, $X \in T_z \Sigma$ and $\xi, \eta \in E_z$,
$$
\langle \xi , B\eta(X) \rangle_\RR = \langle B\xi(X) , \eta \rangle_\RR.
$$
\end{prop}

The remaining crucial ingredient will be a generalization of Taubes's
analytic perturbation theory argument described in \S\ref{sec:Taubes}.
Fix $B : E \to F$ as given by Proposition~\ref{prop:symmetry}, and
assume $\varphi : (\widetilde{\Sigma},\tilde{\jmath}) \to (\Sigma,j)$ 
is a holomorphic map of degree $d \ge 1$.  The generalized normal
bundle of $\tilde{u} := u \circ \varphi$ is 
then $\widetilde{E} := N_{\tilde{u}} = \varphi^*E$, and we define
$\widetilde{F} := \overline{\Hom}_\CC(T\widetilde{\Sigma},\widetilde{E})$
so that $\mathbf{D}_{\tilde{u}}^N$ maps $\Gamma(\widetilde{E})$ to $\Gamma(\widetilde{F})$.
If $\{ J_\tau \}$ is a $1$-parameter family of almost complex
structures as in Proposition~\ref{prop:nonlinear} so that
$\mathbf{D}_{u,\tau}^N$ for each $\tau$ is conjugate to 
$\mathbf{D}_u^N + \tau B$, then the resulting perturbed
normal Cauchy-Riemann operators $\mathbf{D}_{\tilde{u},\tau}^N$
are conjugate to the family
$$
\mathbf{D}_{\tilde{u}}^N + \tau B_\varphi, : \Gamma(\widetilde{E}) \to
\Gamma(\widetilde{F}),
$$
where
$$
B_\varphi : \varphi^*E \to \overline{\Hom}_\CC(T\widetilde{\Sigma},\varphi^*E) :
\eta \mapsto B\eta \circ T\varphi.
$$
We will prove the following in \S\ref{sec:linear}\rev{, using a Weitzenb\"ock
formula developed in \S\ref{sec:Weitzenbock}}.

\begin{prop}
\label{prop:linear}
Given any $B$ and $\varphi$ as described above, the operator
$\mathbf{D}_{\tilde{u}}^N + \tau B_\varphi$ 
is injective for all $\tau \in \RR$ outside of a discrete subset.
\end{prop}

\subsection{\rev{Proof of Theorem~\ref{thm:super}}}

\rev{Assuming Propositions~\ref{prop:nonlinear}, \ref{prop:symmetry} 
and~\ref{prop:linear}, we now prove the main result.}
The following topological argument is also inspired by ideas of Taubes
(cf.~\cite{McDuffSalamon:Jhol}*{pp.~52--53} or \cite{Wendl:lecturesV2}*{\S 4.4.2}).
We shall carry out the argument first in the setting of embedded holomorphic
curves and compatible almost complex structures, and then explain what
modifications are needed for the immersed/tame case.

Fix \rev{an integer $g \ge 0$,
a homology class $A \in H_2(M)$ and a
closed connected and oriented surface $\Sigma$ of genus~$g$.
Recall that the Teichm\"uller space $\tT(\Sigma) = 
\jJ(\Sigma) / \Diff_0(\Sigma)$ is
a smooth manifold diffeomorphic to $\CC^N$, with
$N = 3g-3$ for $g \ge 2$ or $N = g$ for $g = 0,1$.  In particular,
$\tT(\Sigma)$ is contractible, allowing us to fix a smooth family of
complex structures
$$
\{ j_x \in \jJ(\Sigma) \}_{x \in \CC^N}
$$
for which the natural projection to $\tT(\Sigma)$ is bijective.
Fix Riemannian metrics on $\Sigma$ and $M$, denoting
the resulting distance functions all by $\dist(\ ,\ )$.  
Now for any $J \in \jJ(M\,;\,\uU,\Jfix)$ and $N \in \NN$, define
$$
\mM_g(A,J,N) \subset \mM_{g,0}(A,J)
$$
to consist of every equivalence class in $\mM_{g,0}(A,J)$ admitting a
representative of the form $(\Sigma,j_x,u)$ such that the following conditions are
satisfied:}
\begin{enumerate}
\item \rev{$j_x$ is} ``not close to degenerating'':
$$
\rev{|x| \le N}
$$
\item $u$ is ``not close to bubbling'':
$$
| du(z) | \le N \quad\text{ for all $z \in \Sigma$};
$$
\item $u$ is ``not close to being non-embedded'':
$$
\min_{z \in \Sigma} | du(z) | \ge \frac{1}{N} ,\quad\text{ and }\quad
\inf_{z,\zeta \in \Sigma,\ z \ne \zeta} \frac{\dist(u(z),u(\zeta))}{\dist(z,\zeta)}
\ge \frac{1}{N};
$$
\item $u$ is ``not close to escaping~$\uU$'':
$$
\dist\left(u(\Sigma), M \setminus \uU\right) \ge \frac{1}{N}.
$$
\end{enumerate}
The union of the subsets $\mM_{g}(A,J,N)$ for all $N \in \NN$ consists
precisely of all curves in $\mM_{g,0}(A,J)$ that are embedded and
contained in~$\uU$.  \rev{We claim that for any fixed $N \in \NN$,
$\mM_g(A,J,N)$ is compact---in fact:}

\rev{
\begin{lemma}
For any $N \in \NN$ and any convergent sequence $J_k \to J
\in \jJ(M\,;\,\uU,\Jfix)$, every sequence
$u_k \in \mM_g(A,J_k,N)$ has a subsequence converging to an
element of $\mM_g(A,J,N)$.
\end{lemma}
\begin{proof}
By assumption, the given sequence admits representatives
of the form $(\Sigma,j_{x_k},u_k)$ that each satisfy the four conditions
listed above.  Condition~(1) implies $|x_k| \le N$ for all~$k$, so we can
take a subsequence for which the complex structures $j_{x_k}$
converge to some~$j_x$ with $|x| \le N$.  The second condition then implies 
via elliptic regularity that
after passing to a further subsequence, the maps $u_k$ converge
in $C^\infty$ to a pseudoholomorphic map
$u : (\Sigma,j_x) \to (M,J)$ with $|du| \le N$ everywhere.
Given this convergence, (3) and~(4) are both closed conditions
and are thus also satisfied by~$u$, so 
$(\Sigma,j_x,u)$ represents an element of $\mM_g(A,J,N)$.
\end{proof}
}

Now for each $N \in \NN$, define
$$
\jJ_\reg(N) \subset \jJ\comp(M,\omega\,;\,\uU,\Jfix)
$$
to consist of all $J \in \jJ\comp(M,\omega\,;\,\uU,\Jfix)$ with the property 
that for every index~$0$ curve $[(\Sigma,j,u)] \in \mM_{g}(A,J,N)$ and
every unbranched holomorphic cover
$\varphi : (\widetilde{\Sigma},\tilde{\jmath}) \to (\Sigma,j)$ of
degree at most~$N$, the curve $\tilde{u} = u \circ \varphi$ is Fredholm
regular.

We claim that $\jJ_\reg(N)$ is open.  If this is not the case, 
then there exists a sequence
$J_k \in \jJ\comp(M,\omega\,;\,\uU,\Jfix)$ converging to $J \in \jJ_\reg(N)$,
together with a sequence $[(\Sigma,j_k,u_k)] \in \mM_{g}(A,J_k,N)$ and
unbranched covers $\varphi_k : (\widetilde{\Sigma}_k,\tilde{\jmath}_k)
\to (\Sigma,j_k)$ with $\deg(\varphi_k) \le N$
for which $\ind(u_k) = 0$ but $u_k \circ \varphi_k$ is not regular.  But then 
$[(\Sigma,j_k,u_k)]$ has a subsequence converging to an element 
$[(\Sigma,j,u)] \in \mM_{g}(A,J,N)$, and since each $(\Sigma,j_k)$ has
only finitely many unbranched covers of degree at most~$N$ up to
biholomorphic equivalence, we may also assume after reparametrization
that a subsequence of $\varphi_k$ converges to another unbranched cover
$\varphi : (\widetilde{\Sigma},\tilde{\jmath}) \to (\Sigma,j)$ of degree
at most~$N$.  Since
$J \in \jJ_\reg(N)$, $u \circ \varphi$ is regular, but this condition is
open and thus gives
a contradiction.

We claim next that $\jJ_\reg(N)$ is dense.  To see this, note first that by the
standard transversality theory as in \cite{McDuffSalamon:Jhol}, any
$J \in \jJ\comp(M,\omega\,;\,\uU,\Jfix)$ has a perturbation  
\rev{$J' \in \jJ\comp(M,\omega\,;\,\uU,\Jfix)$}
for which all curves in \rev{$\mM_{g}(A,J',N)$ are} Fredholm regular, as all of
them have injective points mapped into~$\uU$.  Since \rev{$\mM_{g}(A,J',N)$} 
is compact, the set of index~$0$ curves in
\rev{$\mM_{g}(A,J',N)$ is now} finite.  
For each individual such curve
$[(\Sigma,j,u)]$ and each unbranched cover
$\varphi : (\widetilde{\Sigma},\tilde{\jmath}) \to (\Sigma,j)$, the 
combination of Propositions~\ref{prop:nonlinear}, 
\ref{prop:symmetry} and \ref{prop:linear} provides a $1$-parameter family of
perturbed almost complex structures 
$\{J_\tau \in \jJ\comp(M,\omega\,;\,\uU,\Jfix)\}$ \rev{with $J_0 = J'$} such that 
the normal Cauchy-Riemann operator of $u \circ \varphi$
becomes injective for sufficiently small $\tau > 0$. \rev{Note that by the
implicit function theorem, there is a natural bijective correspondence
between the sets of index~$0$ curves in $\mM_g(A,J',N)$ and $\mM_g(A,J_\tau,N)$
for $\tau$ sufficiently small.  Now since the set of
covers $u \circ \varphi$ with $u \in \mM_{g}(A,J',N)$, $\ind(u) = 0$
 and $\deg(\varphi) \le N$
is finite up to biholomorphic equivalence, one can repeat this procedure
finitely many times to obtain an arbitrarily small perturbation 
$J''$ of $J'$ for which all such covers become regular,
meaning $J'' \in \jJ_\reg(N)$}.

Finally, the desired Baire subset can be defined as the countable intersection
of the sets $\jJ_\reg(N)$ for all possible $N \in \NN$, $g \ge 0$ and $A \in H_2(M)$,
thus concluding the proof of Theorem~\ref{thm:super} for embedded curves.

\begin{remark}
\label{remark:fail}
The difficulty in using this method to prove super-rigidity for 
\emph{branched} covers is that for a given $(\Sigma,j)$ and $N \in \NN$,
the set of inequivalent branched covers of $(\Sigma,j)$ with degree at most~$N$
is generally uncountable, so there is no guarantee that any single perturbation
$J_\tau$ could make the normal operator injective for all of them at once.
The analytic perturbation trick unfortunately provides no obvious control
over the function
$$
\varphi \mapsto \sup \left\{ \tau_0 > 0\ |\ 
\text{$\mathbf{D}_{u \circ \varphi}^N$ defined with respect to
$J_\tau$ is injective for all $\tau \in (0,\tau_0]$} \right\},
$$
e.g.~it could vary discontinuously as $\varphi$ moves in the moduli
space of branched covers.
\end{remark}

The above argument could also be repeated verbatim to find corresponding
Baire subsets of $\jJ(M\,;\,\uU,\Jfix)$ and $\jJ\tame(M,\omega\,;\,\uU,\Jfix)$
that establish regularity for unbranched covers of embedded curves.  
This means \emph{all} 
simple curves without loss of generality if $\dim_\RR M \ge 6$,
but a modified argument is needed
in dimension four to handle curves with self-intersections.  If
$\dim_\RR M = 4$, we modify the definition of $\mM_{g}(A,J,N)$ as follows.
For any simple curve $u \in \mM_{g,0}(A,J)$, define the integer $d(u) \ge 0$
by
$$
2 d(u) = \big| \left\{ (z,\zeta) \in \Sigma \times \Sigma\ |\ 
\text{$u(z) = u(\zeta)$ and $z \ne \zeta$} \right\} \big|.
$$
Recall that by the adjunction inequality, this number satisfies
$$
A \cdot A \ge 2 d(u) + c_1(A) - (2 - 2g),
$$
with equality if and only if $u$ is immersed with only transverse double
points.  With this in mind, define
$$
d(A,g) := \frac{1}{2}\left( A \cdot A - c_1(A) \right) + 1 - g,
$$
and define $\mM_{g}(A,J,N)$ via conditions (1), (2) and (4) above, plus
the following replacement of condition~(3):
\begin{enumerate}
\renewcommand{\labelenumi}{(3\alph{enumi})}
\item $\displaystyle \min_{z \in \Sigma} | du(z) | \ge \frac{1}{N}$;
\item There exists a point $z_0 \in \Sigma$ such that
$$
\inf_{z \in \Sigma\setminus\{z_0\}} \frac{\dist(u(z_0),u(z))}{\dist(z_0,z)}
\ge \frac{1}{N};
$$
\item $M$ contains $d := d(A,g)$ distinct points $p_1,\ldots,p_d \in M$ 
at which $| u^{-1}(p_j) | > 1$, and
$$
\dist\left( (p_1,\ldots,p_d) , \Delta \right) \ge \frac{1}{N},
$$
where $\Delta \subset M^d$ denotes the set of tuples $(x_1,\ldots,x_d)$ for
which at least two of the points coincide.
\end{enumerate}
The adjunction inequality implies that every curve in 
$u \in \mM_{g}(A,J,N)$ is immersed with transverse double points, all at
distinct points in the image, and 
$\bigcup_{N \in \NN} \mM_{g}(A,J,N)$ now consists of all curves in
$\mM_{g,0}(A,J)$ that have these properties.
The only other modification needed from the embedded case is in the proof that
$\jJ_\reg(N)$ is dense.  This is where we need to allow
$J \in \jJ\tame(M,\omega\,;\,\uU,\Jfix)$ instead of $\jJ\comp(M,\omega\,;\,\uU,\Jfix)$,
as Proposition~\ref{prop:nonlinear} does not provide an $\omega$-compatible
perturbation if $u$ has double points.  Note however that after a small 
perturbation of any given $J$, we are free to assume that all simple index~$0$
curves are immersed with transverse double points at separate points in 
the image (see \rev{e.g.~\cite{Wendl:lecturesV2}*{Exercise~4.65 and \S 4.6}}), in which case 
Propositions~\ref{prop:nonlinear} and~\ref{prop:linear}
can be used to find an $\omega$-tame perturbation in $\jJ_\reg(N)$.
With this established, the rest of the proof goes through as before.
\qed

\section{Normal perturbations of almost complex structures}
\label{sec:nonlinear}

The purpose of this section is to prove Proposition~\ref{prop:nonlinear}.
Fix a tame almost complex structure
$J \in \jJ\tame(M,\omega\,;\,\uU,\Jfix)$ and a closed $J$-holomorphic curve
$u : (\Sigma,j) \to (M,J)$ that has image in~$\uU$ and is immersed with
at most finitely many double points, all transverse and at distinct points
in the image.  Note that if $\dim_\RR M \ge 6$, this assumption means
$u$ is embedded.  

\rev{Choose a complex subbundle $N_u \subset u^*TM$ such that
$u^*TM = T_u \oplus N_u$, where $T_u := \im du$.  In the $4$-dimensional
case, our assumption about double points implies that we can also arrange
$$
(T_u)_z = (N_u)_\zeta \quad\text{ and }\quad (T_u)_\zeta = (N_u)_z
$$
whenever $u(z) = u(\zeta)$ with $z \ne \zeta$.  To construct
a suitable perturbation of $J$, fix $Y \in \Gamma(\overline{\End}_\CC(TM,J))$ 
with support in $\overline{\uU}$ and let
$$
\Phi := \1 + \frac{1}{2} J Y \in \Gamma(\End_\RR(TM)).
$$
We shall always assume that $Y$ is $C^0$-small enough for $\Phi$ to be
everywhere invertible, in which case
$$
J' := \Phi J \Phi^{-1}
$$
defines an almost complex structure that is close to $J$ and therefore
tame if $Y$ is sufficiently small.  
We shall make use of the splitting $u^*TM = T_u \oplus N_u$ and
restrict $Y$ by assuming that along $u$, it takes the block form
\begin{equation}
\label{eqn:Y}
Y(u(z)) = \begin{pmatrix}
0 & Y^{NT}(z) \\
0 & 0
\end{pmatrix} \in \overline{\End}_\CC(T_u \oplus N_u)
\quad\text{ for all $z \in \Sigma$},
\end{equation}
where $Y^{NT}$ is a (necessarily complex-antilinear)
bundle map $N_u \to T_u$.  Note that if $u$ has any double
points, then this condition requires $Y$ to vanish at the images
of those points.
Writing the tangent and normal parts of $J$ along $u$ as
$J^T : T_u \to T_u$ and $J^N : N_u \to N_u$ respectively, we now have
\begin{equation}
\label{eqn:Phitau}
\Phi(u(z)) = \begin{pmatrix}
\1 & \frac{1}{2} J^T(z) Y^{NT}(z) \\
0 & \1
\end{pmatrix}
\quad\text{ for all $z \in \Sigma$},
\end{equation}
and thus
\begin{equation}
\label{eqn:Jtau}
J'(u(z)) = \begin{pmatrix}
J^T(z) & Y^{NT}(z) \\
0   & J^N(z)
\end{pmatrix}
\quad\text{ for all $z \in \Sigma$.}
\end{equation}
This shows that $J'|_{T_u} = J|_{T_u}$, 
so $u$ is also $J'$-holomorpic.  
We can now define a $J'$-invariant normal bundle along $u$ by
$$
N_u' := \Phi(N_u) \subset u^*TM,
$$
so $\Phi|_{N_u} : (N_u,J) \to (N_u',J')$ is a
complex bundle isomorphism by construction.  Let
$\pi_{N'} : u^*TM = T_u \oplus N_u' \to N_u'$ 
denote the resulting normal projection, which gives rise to a perturbed normal
Cauchy-Riemann operator
$$
\mathbf{D}_u^{N'} = \left. \pi_{N'} \circ \mathbf{D}_u'\right|_{\Gamma(N_u')} :
\Gamma(N_u') \to \Omega^{0,1}(\Sigma,N_u'),
$$
where $\mathbf{D}_u'$ denotes the linearized Cauchy-Riemann operator for
$u$ as a $J'$-holomorphic curve.  Conjugating this with the bundle isomorphism
gives an operator
$$
\Phi^{-1} \circ \mathbf{D}_u^{N'} \circ \Phi :
\Gamma(N_u) \to \Omega^{0,1}(\Sigma,N_u).
$$

\begin{lemma}
\label{lemma:zerothOrder}
There exists a smooth bundle map $A : N_u \to \overline{\Hom}_\CC(T\Sigma,N_u)$
such that $\Phi^{-1} \circ \mathbf{D}_u^{N'} \circ \Phi = \mathbf{D}_u^N + A$.
For any connection $\nabla$ on~$TM$, $A$ is given by the formula
$$
A \eta = \pi_N \circ \nabla_\eta Y \circ Tu \circ j.
$$
\end{lemma}
\begin{remark}
Implicit in the above statement is that the expression on the right hand
side of the formula 
does not depend on the choice of connection.  This will follow from
a direct calculation in the proof, but the intuitive reason for it is that
under the block decomposition of $\nabla_\eta Y$ given by the splitting
$u^*TM = T_u \oplus N_u$, only the lower-left block (mapping $T_u$
to $N_u$) is relevant in the
above expression, while the corresponding block of $Y$ itself has been assumed
to vanish along~$u$.
\end{remark}
\begin{proof}[Proof of Lemma~\ref{lemma:zerothOrder}]
In terms of the splitting $u^*TM = T_u \oplus N_u$, the perturbed normal 
projection $u^*TM \to N_u'$ is given in block form by
$$
\pi_{N'} = \begin{pmatrix} 0 & \frac{1}{2} J^T Y^{NT} \\ 0 & \1 \end{pmatrix},
$$
so using \eqref{eqn:Phitau} to write
$\Phi^{-1}(u(z)) = 
\begin{pmatrix} \1 & -\frac{1}{2} J^T(z) Y^{NT}(z) \\ 0 & \1 \end{pmatrix}$,
we find
$$
\Phi^{-1} \circ \pi_{N'} = \pi_N.
$$
Recall now from \cite{Wendl:automatic}*{Lemma~3.8} that $\mathbf{D}_u$ 
maps sections of $T_u$ to $(0,1)$-forms valued in $u^*TM$ with vanishing
normal component.
The same applies to $\mathbf{D}_u'$, hence for $\eta \in \Gamma(N_u)$, we have
$\Phi\eta - \eta \in \Gamma(T_u)$ and thus
$$
\left(\Phi^{-1} \circ \mathbf{D}_u^{N'} \circ \Phi\right) \eta =
(\Phi^{-1} \circ \pi_{N'}) \mathbf{D}_u' (\Phi\eta) =
\pi_N (\mathbf{D}_u' \eta).
$$
To compute $\mathbf{D}_u' \eta$, choose any smooth $1$-parameter family of
maps $u_\rho : \Sigma \to M$ for $\rho \in (-\epsilon,\epsilon)$ with
$u_0 = u$ and $\p_\rho u_\rho|_{\rho=0} = \eta$.  Then for any connection
$\nabla$ on $TM$ and any holomorphic local coordinate system $(s,t)$ on some
open subset in $\Sigma$, the $(0,1)$-form $\mathbf{D}_u' \eta$ is given
locally by
\begin{equation}
\label{eqn:bigThing}
\begin{split}
(\mathbf{D}_u' \eta) \p_s &= \left.\nabla_\rho\left( \p_s u_\rho + 
J'(u_\rho)\, \p_t u_\rho\right)\right|_{\rho=0} \\
&= \left.\nabla_\rho\left( \p_s u_\rho + J(u_\rho)\, \p_t u_\rho +
\left[ J'(u_\rho) - J(u_\rho) \right] \p_t u_\rho \right)\right|_{\rho=0} \\
&= (\mathbf{D}_u \eta) \p_s + \left.\nabla_\rho \left(
\left[ J'(u_\rho) - J(u_\rho) \right] \p_t u_\rho \right)\right|_{\rho=0} \\
&= (\mathbf{D}_u \eta) \p_s + \left[\nabla_\eta(J'-J)\right] \p_t u +
\left[J'(u) - J(u)\right] \left.\nabla_\rho \p_t u_\rho\right|_{\rho=0}.
\end{split}
\end{equation}
By \eqref{eqn:Jtau}, the image of $J' - J$ has vanishing normal component
everywhere along~$u$, so the third term on the right hand side of
\eqref{eqn:bigThing} does not contribute to
$\pi_N(\mathbf{D}_u' \eta)$.  Removing the local coordinates, we thus
obtain the global expression
$$
\left(\Phi^{-1} \circ \mathbf{D}_u^{N'} \circ \Phi\right) \eta =
\mathbf{D}_u^N \eta + \pi_N \circ \nabla_\eta (J' - J) \circ Tu \circ j.
$$
To simplify the last term, observe that since $J' = \Phi J \Phi^{-1}$
with $\Phi = \1 + \frac{1}{2} JY$, $JY = -YJ$ and $J^2 = -\1$, we have
$$
(J'-J) \Phi = \Phi J - J\Phi = \left(\1 + \frac{1}{2} JY \right) J -
J \left(\1 + \frac{1}{2} JY \right) 
= \frac{1}{2} JYJ + \frac{1}{2} Y = Y,
$$
hence $J' - J = Y \Phi^{-1}$, and therefore
$$
\nabla_\eta (J' - J) = (\nabla_\eta Y) \Phi^{-1} + Y (\nabla_\eta \Phi^{-1}).
$$
Composing the second of these two terms with $Tu \circ j$ produces a
section with vanishing normal component due to \eqref{eqn:Y}, 
so it does not contribute.  In the
remaining expression, $\Phi^{-1}$ can be omitted since it acts trivially
on the tangential component, and this produces the formula that was claimed.
\end{proof}

\begin{proof}[Proof of Proposition~\ref{prop:nonlinear}]
Given a bundle map $B : N_u \to \overline{\Hom}_\CC(T\Sigma,N_u)$,
it will suffice to carry out the construction in Lemma~\ref{lemma:zerothOrder}
with $\Phi$ replaced by the $1$-parameter family of bundle isomorphisms
$\Phi_\tau = \1 + \frac{1}{2} \tau J Y$, as long as
$Y \in \Gamma(\overline{\End}_\CC(TM,J))$ can be chosen to match a
block expression of the form \eqref{eqn:Y} along $u$, with normal
derivative along $u$ satisfying
\begin{equation}
\label{eqn:normalDeriv}
\pi_N \circ \nabla_\eta Y \circ Tu \circ j = B \eta \quad\text{ for all $\eta \in N_u$}.
\end{equation}
Since $Tu \circ j : T\Sigma \to T_u$ is a complex-linear bundle isomorphism,
this is clearly possible if $u$ is embedded, as one can then assume
$Y = 0$ along $u$ and choose its normal derivative to satisfy
\eqref{eqn:normalDeriv}.  Note that if $J$ is $\omega$-compatible,
then $J_\tau$ will also be $\omega$-compatible if and only if $Y$ is
everywhere symmetric with respect to the metric $\omega(\cdot,J\cdot)$,
and this can also be achieved in the absence of double points since
\eqref{eqn:normalDeriv} only constrains the lower-left block of
$\nabla_\eta Y$ with respect to the splitting $u^*TM = T_u \oplus N_u$.

We must be a bit more careful 
if $\dim_\RR M = 4$ and $u$ has double points.  Assume
$u(z) = u(\zeta) = p$, with $(T_u)_z = (N_u)_\zeta$ and vice versa.
We can choose local coordinates $(z_1,z_2) \in \CC^2$ near $p$ that
identify $p$ with the origin, while the images of $u$ near $z$ and $\zeta$
are identified with subsets of $\CC \times \{0\}$ and $\{0\} \times \CC$
respectively.  In this neighborhood, choose a complex local trivialization of
$(TM,J)$ identifying the normal subspaces along $\CC \times \{0\}$ with
$\{0\} \oplus \CC$ and those along $\{0\} \times \CC$ with $\CC \oplus \{0\}$,
and let $\nabla$ be
the trivial connection with respect to this trivialization.  We claim that
in this trivialization near~$p$, a suitable $Y$ can be written in the form
$$
Y(z_1,z_2) = \begin{pmatrix}
0 & Y_{12}(z_1,z_2) \\
Y_{21}(z_1,z_2) & 0
\end{pmatrix}
$$
for some functions $Y_{12}$ and $Y_{21}$ valued in $\overline{\End}_\CC(\CC)$.
Indeed, the condition \eqref{eqn:Y} now becomes
\begin{equation*}
\begin{split}
Y_{21}(z_1,0) = 0 \quad &\text{ for all $z_1$},\\
Y_{12}(0,z_2) = 0 \quad &\text{ for all $z_2$},
\end{split}
\end{equation*}
while \eqref{eqn:normalDeriv} specifies the normal derivatives of
$Y_{21}$ along $\CC \times \{0\}$ and $Y_{12}$ along $\{0\} \times \CC$.
After choosing $Y_{12}$ and $Y_{21}$ to satisfy these conditions,
we can then also arrange
$Y_{21}(0,z_2) = Y_{12}(z_1,0) = 0$ for all $z_1,z_2$ ouside some small
neighborhood of~$0$, hence $Y$ vanishes along $u$ outside a neighborhood
of~$p$, and the previous argument for the embedded case can then be used
to extend $Y$ globally.
\end{proof}

\begin{remark}
\label{remark:whyNotCompatible}
If $J$ is $\omega$-compatible and $u$ has double points, then the above
proof fails to provide $\omega$-compatible perturbations~$J_\tau$:
in a neighborhood of a double point, the last step in the
construction  generally forces the upper-right block of \eqref{eqn:Y} to take
nonzero values, thus violating the symmetry condition required for
$\omega$-compatibility.
This is why the statement of Theorem~\ref{thm:super} in the 
compatible case is limited to embedded curves.
\end{remark}
}

\section{Symmetric bundle isomorphisms}
\label{sec:symmetry}

We now state and prove a result that implies Proposition~\ref{prop:symmetry}.  

\begin{prop}
\label{prop:symmetryGeneral}
Suppose $E \to \Sigma$ is a Hermitian vector bundle,
let $\langle\ ,\ \rangle_\RR$ denote the real part of its bundle metric,
and suppose $L \to \Sigma$ is a complex line bundle.  Then every homotopy 
class of complex-antilinear bundle isomorphisms 
$B : E \to \overline{\Hom}_\CC(L,E)$ contains one that satisfies the
condition
$$
\langle \xi , B\eta(X) \rangle_\RR = \langle B\xi(X) , \eta \rangle_\RR
\quad\text{ for all $(X,\xi,\eta) \in L \oplus E \oplus E$}.
$$
\end{prop}

Observe first that a choice of {complex-antilinear isomorphism}
$B : E \to \overline{\Hom}_\CC(L,E)$ is equivalent via the correspondence
$B\eta(X) = \widehat{B}X(\eta)$ to a choice of
complex-antilinear bundle map
$$
\widehat{B} : L \to \overline{\End}_\CC(E)
$$
with the property that for all nonzero $X \in L$,
$\widehat{B}(X)$ is invertible.  Proposition~\ref{prop:symmetryGeneral} is then 
equivalent to showing that 
every homotopy class of bundle maps $\widehat{B}$ with the above property
contains one for which $\widehat{B}(X)$ is always symmetric.
This is clearly true for the restriction of $\widehat{B}$ to the 
$0$-skeleton of $\Sigma$, since the space of antilinear isomorphisms
on any complex vector space is connected and contains one that is
symmetric.  Extending this to the $1$-skeleton 
and then the $2$-skeleton of $\Sigma$ is possible due to 
Proposition~\ref{prop:obstruction} below.

Identify $\CC^m$ with $\RR^{2m}$ so that $\End_\CC(\CC^m)$ is regarded as
the real subspace of $\End_\RR(\RR^{2m}) = \End_\RR(\CC^m)$ consisting of
linear maps that commute with the standard complex structure
$i \in \GL(2m,\RR)$.  We then denote
\begin{equation*}
\begin{split}
\overline{\Aut}_\CC(\CC^m) &:= \overline{\End}_\CC(\CC^m) \cap \GL(2m,\RR),\\
\overline{\Aut}_\CC^S(\CC^m) &:= \left\{ A \in \overline{\Aut}_\CC(\CC^m)\ |\ 
A = A^T \right\},
\end{split}
\end{equation*}
where $A^T$ means the usual transpose of real $2m$-by-$2m$ matrices.

\begin{prop}
\label{prop:obstruction}
We have
$$
\pi_1\left(\overline{\Aut}_\CC(\CC^m),\overline{\Aut}_\CC^S(\CC^m)\right) =
\pi_2\left(\overline{\Aut}_\CC(\CC^m),\overline{\Aut}_\CC^S(\CC^m)\right) = 0.
$$
\end{prop}

The proof of the proposition occupies the remainder of this section.
Observe first that composition with the real-linear isomorphism
$$
\CC^m \to \CC^m : v \mapsto \bar{v}
$$
identifies $\overline{\Aut}_\CC(\CC^m)$ with
$\GL(m,\CC) \subset \GL(2m,\RR)$ and $\overline{\Aut}_\CC^S(\CC^m)$ with
$$
\GL^S(m,\CC) := \left\{ A \in \GL(m,\CC)\ |\ A = A^T \right\},
$$
where in the latter case $A^T$ denotes the transpose (not the adjoint!)
of the $m$-by-$m$ complex matrix~$A$, i.e.~$A^T = \overline{A}^\dagger$.
The proposition is therefore equivalent to the computation
\begin{equation}
\label{eqn:obstruction}
\pi_1\left(\GL(m,\CC),\GL^S(m,\CC)\right) = 
\pi_2\left(\GL(m,\CC),\GL^S(m,\CC)\right) = 0.
\end{equation}
We prove this in five steps.

\textsl{Step~1}.
Consider the map
\begin{equation}
\label{eqn:ATA}
Q : \GL(m,\CC) / \Ortho(m,\CC) \to \GL^S(m,\CC) : A \mapsto A^T A,
\end{equation}
where $\Ortho(m,\CC)$ denotes the complex orthogonal group
$\{ A \in \GL(m,\CC)\ |\ A^T A = \1 \}$.
We claim that $Q$ is a bijection.  Injectivity is easy to check;
surjectivity follows from the fact that every $A \in \GL^S(m,\CC)$ defines
a symmetric nondegenerate complex bilinear form
$$
(v,w) \mapsto v^T A w,
$$
and all such forms are equivalent up to a choice of basis.
Since $\GL(m,\CC)$ is connected, it follows that $\GL^S(m,\CC)$ is connected.

\textsl{Step~2}.
We claim that for all $m \in \NN$, $\Ortho(m,\CC)$ has exactly two
connected components.  It is clear that there are at least two,
as every $A \in \Ortho(m,\CC)$ has $\det A = \pm 1$.  It suffices
therefore to prove that $\SO(m,\CC) := \{ A \in \Ortho(m,\CC)\ |\ \det A = 1 \}$
is connected.  This is true for $m=1$ since $\SO(1,\CC)$ is the trivial group.
The claim then follows by induction using the fibration
$$
\SO(m-1,\CC) \hookrightarrow \SO(m,\CC) \stackrel{\pi}{\to} H^{m-1},
$$
where $H^{m-1} := \{ v \in \CC^m\ |\ v^T v = 1 \}$ and $\pi(A)$ is defined as
the first column of~$A$.  The fact that $\pi$ is surjective can be proved
using the same argument that is used in diagonalizing quadratic forms:
it reduces to the fact that any given $v_1 \in H^{m-1}$ can be extended to a
complex basis $v_1,\ldots,v_m \in H^{m-1}$ of $\CC^m$ such that
$v_i^T v_j = \delta_{ij}$.

\textsl{Step~3}.
We claim that $\pi_1(\GL(m,\CC) / \Ortho(m,\CC)) \cong \ZZ$ is generated by
the projection to $\GL(m,\CC) / \Ortho(m,\CC)$ of the path
$$
\gamma : [0,1] \to \GL(m,\CC) : t \mapsto
\begin{pmatrix}
e^{\pi i t} &   &        &   \\
            & 1 &        &   \\
	    &   & \ddots &   \\
	    &   &        & 1
\end{pmatrix}.
$$
To see this, consider the long exact sequence of the fibration
$\Ortho(m,\CC) \stackrel{\iota}{\hookrightarrow} \GL(m,\CC) 
\stackrel{p}{\to} GL(m,\CC) / \Ortho(m,\CC)$:
\begin{equation*}
\begin{split}
\ldots \longrightarrow \pi_1(\GL(m,\CC)) &\stackrel{p_*}{\longrightarrow} 
\pi_1(\GL(m,\CC) / \Ortho(m,\CC)) \stackrel{\p}{\longrightarrow} \\
&\qquad \pi_0(\Ortho(m,\CC)) \longrightarrow
\pi_0(\GL(m,\CC)) = 0.
\end{split}
\end{equation*}
Any loop in $\GL(m,\CC) / \Ortho(m,\CC)$ can be represented as a path
$\beta : [0,1] \to \GL(m,\CC)$ with $\beta(0) = \1$ and
$\beta(1) \in \Ortho(m,\CC)$, and the map $\p$ can then be written as
$$
\p[\beta] = \det \beta(1) \in \{1,-1\} = \pi_0(\Ortho(m,\CC)),
$$
applying the result of Step~2.
Since $\ker \p = \im p_*$, any such path $\beta$ with $\det\beta(1) = 1$
is equivalent in $\pi_1(\GL(m,\CC) / \Ortho(m,\CC))$ to a loop in
$\GL(m,\CC)$, and using the standard computation of $\pi_1(\GL(m,\CC)) = 
\pi_1(\U(m))$, any such loop is homotopic to
$$
S^1 \to \GL(m,\CC) : t \mapsto
\begin{pmatrix}
e^{2\pi k i t} &   &        &   \\
             & 1 &        &   \\
 	     &   & \ddots &   \\
	     &   &        & 1
\end{pmatrix}
$$
for some $k \in \ZZ$.
Thus any such element of $\pi_1(\GL(m,\CC) / \Ortho(m,\CC))$ is an even
power of~$\gamma$.  If on the other hand $\det \beta(1) = -1$, then we
can concatenate $\beta$ with the loop $t \mapsto [\beta(1) \gamma(t)]$
in $\GL(m,\CC) / \Ortho(m,\CC)$, whose determinant at $t=1$ is positive,
implying that $\beta \cdot \gamma \in \pi_1(\GL(m,\CC) / \Ortho(m,\CC))$
is an even power of $\gamma$, so this proves the claim.

\textsl{Step~4}.
We claim that the composition of the map $Q$ in \eqref{eqn:ATA} with the
inclusion $\GL^S(m,\CC) \hookrightarrow \GL(m,\CC)$ induces an isomorphism
$$
\pi_1\left( \GL(m,\CC) / \Ortho(m,\CC) \right) = \pi_1(\GL(m,\CC)).
$$
This {follows by} computing the action of this map on the
generator of $\pi_1\left( \GL(m,\CC) / \Ortho(m,\CC) \right)$ as
described in Step~3.

\textsl{Step~5}.
Consider the homotopy exact sequence for $(\GL(m,\CC),\Ortho(m,\CC))$:
\begin{equation*}
\begin{split}
\ldots \longrightarrow &\pi_2(\GL(m,\CC)) \stackrel{\alpha_2}{\longrightarrow} \pi_2\left(\GL(m,\CC),\GL^S(m,\CC)\right)
\stackrel{\p_2}{\longrightarrow} \\
& \pi_1\left(\GL^S(m,\CC)\right) \stackrel{\iota_*}{\longrightarrow} \pi_1(\GL(m,\CC))
\stackrel{\alpha_1}{\longrightarrow} \pi_1\left(\GL(m,\CC) , \GL^S(m,\CC)\right) \stackrel{\p_1}{\longrightarrow}\\
& \pi_0\left(\GL^S(m,\CC)\right) = 0.
\end{split}
\end{equation*}
We showed in Step~4 that $\iota_*$ is an isomorphism, thus $\alpha_1 = 0$,
implying that $\p_1$ is injective and thus
$$
\pi_1\left(\GL(m,\CC) , \GL^S(m,\CC)\right) = 0.
$$
Moreover, the injectivity of $\iota_*$ implies $\p_2 = 0$, so $\alpha_2$
is surjective and, since $\pi_2(\GL(m,\CC)) = \pi_2(\U(m)) = 0$,
$$
\pi_2\left(\GL(m,\CC) , \GL^S(m,\CC)\right) = 0.
$$
This completes the proof of Proposition~\ref{prop:obstruction} and hence,
by standard obstruction theory as in \cite{Steenrod},
Proposition~\ref{prop:symmetryGeneral}.

\rev{

\section{A Weitzenb\"ock formula for antilinear perturbations}
\label{sec:Weitzenbock}

In preparation for the proof of Proposition~\ref{prop:linear}, we now
explain a generalization of the Weitzenb\"ock formula that was derived
in \S\ref{sec:Taubes} for trivial bundles on the torus.  

Throughout this
section, we assume $(\Sigma,j)$ is a closed connected Riemann surface
and $(E,J) \to (\Sigma,j)$ is a complex vector bundle of rank $m \in \NN$
with Hermitian structure $\langle\ ,\ \rangle_E$.  Fix also a
$j$-invariant Riemannian metric on $\Sigma$,
which is the real part of a Hermitian structure $\langle\ ,\ \rangle_\Sigma$
on~$T\Sigma$, and denote the induced volume form on $\Sigma$ by~$\vol$.  
This choice determines a complex-linear bundle 
isomorphism\footnote{We are using the convention that Hermitian
bundle metrics are antilinear in the first and linear in the second argument.}
\begin{equation}
\label{eqn:TSigma}
T\Sigma \to \Lambda^{0,1}T^*\Sigma : X \mapsto X^{0,1} := \langle \cdot,X \rangle_\Sigma
\end{equation}
and consequently a global trivialization
\begin{equation}
\label{eqn:globalTriv}
\Lambda^{1,0}T^*\Sigma \otimes \Lambda^{0,1}T^*\Sigma \to \CC :
\lambda \otimes X^{0,1} \mapsto \lambda(X).
\end{equation}
Moreover, the rank~$m$ complex bundle
$$
F := \Lambda^{0,1}T^*\Sigma \otimes E
$$
inherits from $\langle\ ,\ \rangle_\Sigma$ and $\langle\ ,\ \rangle_E$ a
Hermitian bundle metric $\langle\ ,\ \rangle_F$, and we shall define 
real-valued $L^2$-pairings for sections of $E$ and $F$ by
\begin{equation*}
\begin{split}
\langle \eta,\xi \rangle_{L^2(E)} &:= \Re \int_{\Sigma} 
\langle \eta,\xi \rangle_E \, \vol, \quad \text{ for } \quad
\eta,\xi \in \Gamma(E), \\
\langle \alpha,\lambda \rangle_{L^2(F)} &:= \Re \int_{\Sigma} 
\langle \alpha,\lambda \rangle_F \, \vol, \quad \text{ for } \quad
\alpha,\lambda \in \Gamma(F).
\end{split}
\end{equation*}
Given any real-linear map
$\mathbf{D} : \Gamma(E) \to \Gamma(F)$, the \defin{formal adjoint}
$\mathbf{D}^* : \Gamma(F) \to \Gamma(E)$ is defined via the relation
$$
\langle \lambda , \mathbf{D} \eta \rangle_{L^2(F)} =
\langle \mathbf{D}^* \lambda , \eta \rangle_{L^2(E)} \quad \text{ for all }\quad
\eta \in \Gamma(E),\ \lambda \in \Gamma(F).
$$
Recall that $\mathbf{D} : \Gamma(E) \to \Omega^{0,1}(\Sigma,E) = \Gamma(F)$ 
is called a \defin{Cauchy-Riemann type} operator on $E$
if it satisfies the Leibniz rule
$$
\mathbf{D}(f\eta) = (\dbar f) \eta + f \, \mathbf{D}\eta \quad \text{ for all }
\quad f \in C^\infty(\Sigma,\RR),\ \eta \in \Gamma(E),
$$
where $\dbar f := df + i\, df \circ j$.
Similarly, we will say that $\mathbf{D} : E \to \Omega^{1,0}(\Sigma,E) 
= \Gamma(\Lambda^{1,0}T^*\Sigma \otimes E)$ is an
\defin{anti-Cauchy-Riemann type} operator on $E$ if it satisfies
\begin{equation}
\label{eqn:antiLeibniz}
\mathbf{D}(f\eta) = (\p f) \eta + f \, \mathbf{D}\eta \quad \text{ for all }
\quad f \in C^\infty(\Sigma,\RR),\ \eta \in \Gamma(E),
\end{equation}
with $\p f := df - i\, df \circ j$.  If $\mathbf{D}$ is of Cauchy-Riemann type,
then it is well known that $\mathbf{D}^*$ is conjugate via real-linear
bundle isomorphisms to another Cauchy-Riemann type operator; more
precisely, the natural complex bundle isomorphism
\begin{equation}
\label{eqn:bndlIso}
\Lambda^{1,0}T^*\Sigma \otimes F = \Lambda^{1,0}T^*\Sigma \otimes
\Lambda^{0,1}T^*\Sigma \otimes E = E
\end{equation}
defined via \eqref{eqn:globalTriv} identifies $-\mathbf{D}^*$ with
an anti-Cauchy-Riemann type operator
$$
-\mathbf{D}^* : \Gamma(F) \to \Gamma(E) =
\Gamma(\Lambda^{1,0}T^*\Sigma \otimes F) = \Omega^{1,0}(\Sigma,F).
$$

\begin{prop}
\label{prop:Weitzenbock}
Suppose $\mathbf{D} : \Gamma(E) \to \Gamma(F)$ is a real-linear Cauchy-Riemann 
type operator, $B : E \to F$ is a complex-antilinear bundle map satisfying
the symmetry condition
\begin{equation}
\label{eqn:symmetryCondition}
\Re \langle \eta, B\xi(X) \rangle_E = \Re \langle B\eta(X),\xi\rangle_E
\quad \text{ for all } \quad
(X,\eta,\xi) \in T\Sigma \oplus E \oplus E,
\end{equation}
and $\mathbf{D}_B := \mathbf{D} + B$.  Then the complex vector 
bundle\footnote{We define the complex structure on $\Hom_\RR(E,F)$ and
its subbundles such as $\overline{\Hom}_\CC(E,F)$ via 
the complex structure of $F$, i.e.~$B \mapsto J \circ B$.}
$\overline{\Hom}_\CC(E,F)$ admits a real-linear anti-Cauchy-Riemann type
operator $\p_H$ such that for all $\eta \in \Gamma(E)$,
$$
\mathbf{D}_B^*\mathbf{D}_B \eta =
\mathbf{D}^*\mathbf{D} \eta + B^*B \eta - (\p_H B) \eta.
$$
\end{prop}
\begin{remark}
In the above formula, the product of
$\p_H B \in \Omega^{1,0}(\Sigma,\overline{\Hom}_\CC(E,F))$
with $\eta \in \Gamma(E)$ is interpreted as a section of $E$ via the
product pairing
$$
\left( \Lambda^{1,0}T^*\Sigma \otimes \overline{\Hom}_\CC(E,F) \right)
\otimes E \to \Lambda^{1,0}T^*\Sigma \otimes F
$$
and the isomorphism \eqref{eqn:bndlIso}.
\end{remark}

The proof of Proposition~\ref{prop:Weitzenbock} will rely mainly on a few
basic observations about anti-Cauchy-Riemann operators.
Recall that a complex-valued
function $f$ on an open subset of $\Sigma$ is called \defin{antiholomorphic} 
if it satisfies $\p f \equiv 0$.  The composition of a holomorphic and an
antiholomorphic function is antiholomorphic, and the product of two
antiholomorphic functions is also antiholomorphic, thus it makes sense
to speak of \emph{antiholomorphic vector bundles} over~$\Sigma$.
Anti-Cauchy-Riemann type operators have several properties analogous
to Cauchy-Riemann type operators, notably:
\begin{enumerate}
\item The difference between two anti-Cauchy-Riemann type operators on the
same bundle is a zeroth-order operator.
\item The complex-linear part of any real-linear anti-Cauchy-Riemann type
operator is also an anti-Cauchy-Riemann type operator.
\item Every antiholomorphic vector bundle carries a natural complex-linear
anti-Cauchy-Riemann operator that annihilates local antiholomorphic sections,
and conversely, every complex-linear anti-Cauchy-Riemann operator on
$(E,J) \to (\Sigma,j)$ induces an
antiholomorphic bundle structure in this way.
\end{enumerate}
The first two statements are easy consequences of the Leibniz rule
\eqref{eqn:antiLeibniz}.  The third is nontrivial, but is equivalent to the
corresponding fact about Cauchy-Riemann type operators and holomorphic
bundles over Riemann surfaces.

\begin{lemma}
\label{lemma:induced}
Suppose $E_1$ and $E_2$ are complex vector bundles over $(\Sigma,j)$ endowed with
anti-Cauchy-Riemann type operators $\mathbf{D}_1$ and $\mathbf{D}_2$
respectively.  Then $\Hom_\CC(E_1,E_2)$ admits an anti-Cauchy-Riemann type
operator $\mathbf{D}_{12}$ such that for 
all $\Phi \in \Gamma(\Hom_\CC(E_1,E_2))$ and $\eta \in \Gamma(E_1)$,
$$
\mathbf{D}_2 (\Phi \eta) = (\mathbf{D}_{12} \Phi) \eta + \Phi (\mathbf{D}_1 \eta).
$$
\end{lemma}
\begin{proof}
Write $\mathbf{D}_1 = \mathbf{D}_1^\CC + A$
and $\mathbf{D}_2 = \mathbf{D}_2^\CC + B$, where $\mathbf{D}_1^\CC$ and
$\mathbf{D}_2^\CC$ are complex-linear anti-Cauchy-Riemann type operators
(e.g.~the complex-linear parts of $\mathbf{D}_1$ and $\mathbf{D}_2$ 
respectively), so 
$$
A : E_1 \to \Lambda^{1,0}T^*\Sigma \otimes E_1 \quad \text{ and } \quad
B : E_2 \to \Lambda^{1,0}T^*\Sigma \otimes E_2
$$ 
are zeroth-order terms.  Then $\mathbf{D}_1^\CC$ and $\mathbf{D}_2^\CC$ induce
antiholomorphic bundle structures on $E_1$ and $E_2$, and $\Hom_\CC(E_1,E_2)$
therefore inherits local trivializations with transition maps that are products
of antiholomorphic functions, giving rise to an antiholomorphic structure
and a corresponding complex-linear anti-Cauchy-Riemann operator
$\mathbf{D}_{12}^\CC$ that satisfies
$$
\mathbf{D}_2^\CC (\Phi \eta) = (\mathbf{D}_{12}^\CC \Phi) \eta + \Phi (\mathbf{D}_1^\CC \eta)
$$
for all $\Phi \in \Gamma(\Hom_\CC(E_1,E_2))$ and $\eta \in \Gamma(E_1)$.
The desired operator can then be defined as $\mathbf{D}_{12} = \mathbf{D}_{12}^\CC
+ C$, where $C : \Hom_\CC(E_1,E_2) \to \Lambda^{1,0}T^*\Sigma \otimes \Hom_\CC(E_1,E_2)$
is a bundle map taking the form
$$
(C\Phi) \eta = B(\Phi \eta) - \Phi(A \eta) \in \Lambda^{1,0}T^*\Sigma 
\otimes E_2
$$
for $(\Phi,\eta) \in \Hom_\CC(E_1,E_2) \oplus E_1$.
\end{proof}

For any vector bundle $(E_1,J_1)$ over $\Sigma$, let $E\conjug_1$ denote
its \defin{conjugate bundle}, defined as the same real vector bundle
but with complex structure~$-J_1$.  The identity map gives a natural
complex-antilinear bundle isomorphism
$$
E_1 \to E\conjug_1 : v \mapsto \bar{v},
$$
and if $E_1$ carries a Hermitian bundle metric $\langle\ ,\ \rangle_{E_1}$,
its conjugate inherits a Hermitian structure defined by
$$
\langle \bar{v},\bar{w} \rangle_{E\conjug_1} =
\langle w,v \rangle_{E_1}.
$$
There are canonical complex-linear bundle isomorphisms
$$
(E_1 \otimes E_2)\conjug = E\conjug_1 \otimes E\conjug_2, \quad
\Hom_\CC(E_1,E_2)\conjug = \Hom_\CC(E\conjug_1,E\conjug_2),
\quad
\Hom_\CC(E\conjug_1,E_2) = \overline{\Hom}_\CC(E_1,E_2),
$$
where the third of these identifies $\beta \in \Hom_\CC(E\conjug_1,E_2)$
with the antilinear map 
$$
B : E_1 \to E_2 : \eta \mapsto \beta \bar{\eta}.
$$
The metric on $\Sigma$ determines a complex-linear isomorphism
$$
(T\Sigma)\conjug \to \Lambda^{1,0}T^*\Sigma : \bar{X} \mapsto
X^{1,0} := \langle X,\cdot \rangle_\Sigma,
$$
so together with \eqref{eqn:TSigma}, this
identifies $\Lambda^{1,0}T^*\Sigma$ and $\Lambda^{0,1}T^*\Sigma$
with each other's conjugate bundles.
Observe now that if $\mathbf{D} : \Gamma(E) \to \Gamma(F)$ is a Cauchy-Riemann 
type operator, then 
$$
\mathbf{D}\conjug \bar{\eta} := \overline{\mathbf{D}\eta}
$$
defines an anti-Cauchy-Riemann type operator
$$
\mathbf{D}\conjug : \Gamma(E\conjug) \to \Gamma(F\conjug) =
\Gamma\big((\Lambda^{0,1}T^*\Sigma \otimes E)\conjug \big) =
\Gamma(\Lambda^{1,0}T^*\Sigma \otimes E\conjug) =
\Omega^{1,0}(\Sigma,E\conjug).
$$

Given an antilinear bundle map $B : E \to F$, let
$\beta : E\conjug \to F$ denote the corresponding complex-linear 
bundle map such that
$$
B \eta = \beta \bar{\eta},
$$
and let $\beta^\dagger : F \to E\conjug$ denote the adjoint of
$\beta$ with respect to the Hermitian structures on $E\conjug$ and $F$,
i.e.
$$
\langle \lambda , \beta \bar{\eta} \rangle_F =
\langle \beta^\dagger \lambda , \bar{\eta} \rangle_{E\conjug}
\quad \text{ for all } \quad
(\bar{\eta},\lambda) \in E\conjug \oplus F.
$$
Conjugating this then gives a bundle map
$$
\overline{\beta^\dagger} = \bar{\beta}^\dagger : F\conjug \to E.
$$
We claim that $\beta : E\conjug \to F$ can also be regarded as a
bundle map $F\conjug \to E$.  Indeed, using the isomorphism
$$
F\conjug = (\Lambda^{0,1}T^*\Sigma \otimes E)\conjug = 
\Lambda^{1,0}T^*\Sigma \otimes E\conjug,
$$
we obtain from $\beta : E\conjug \to F$ a bundle map
$$
F\conjug = \Lambda^{1,0}T^*\Sigma \otimes E\conjug 
\stackrel{\1 \otimes \beta}{\longrightarrow}
\Lambda^{1,0}T^*\Sigma \otimes F,
$$
where the target can be identified with $E$
via~\eqref{eqn:bndlIso}.

\begin{lemma}
\label{lemma:symmetry}
Fix a complex-linear bundle map $\beta : E\conjug \to F$ and let
$B : E \to F : \eta \mapsto \beta\bar{\eta}$.  Then $B$ satisfies the
symmetry condition \eqref{eqn:symmetryCondition} if and only if
$\beta$ and $\bar{\beta}^\dagger$ define identical bundle maps
$F\conjug \to E$.
\end{lemma}
\begin{proof}
It will suffice to show that \eqref{eqn:symmetryCondition} holds if and only
if for every $z \in \Sigma$, $\eta \in E_z$ and $\bar{\lambda} \in F\conjug_z$,
$$
\Re \langle \beta \bar{\lambda}, \eta \rangle_E =
\Re \langle \bar{\beta}^\dagger \bar{\lambda}, \eta \rangle_E.
$$
Choose any nonzero vector $X \in T_z \Sigma$; we can then write
$\lambda = X^{0,1} \otimes \xi \in \Lambda^{0,1}T^*_z \Sigma \otimes E_z = F_z$
where $\xi := \lambda(X) / |X|_\Sigma^2 \in E_z$.  Similarly,
$\beta\bar{\eta} = B \eta = X^{0,1} \otimes \theta$, 
where $\theta := B\eta(X) / |X|_\Sigma^2 \in E_z$.   Then
\begin{equation*}
\begin{split}
\langle \bar{\beta}^\dagger \bar{\lambda}, \eta \rangle_E &=
\langle \bar{\lambda} , \bar{\beta} \eta \rangle_{F\conjug} =
\langle \beta \bar{\eta} , \lambda \rangle_F
= \langle X^{0,1} \otimes \theta , X^{0,1} \otimes \xi \rangle_F 
= \langle X,X \rangle_\Sigma \langle \theta,\xi \rangle_E \\
&= \langle B\eta(X) , \xi \rangle_E.
\end{split}
\end{equation*}
Likewise, writing $\beta\bar{\xi} = X^{0,1} \otimes \zeta$ for
$\zeta := B\xi(X) / |X|_\Sigma^2 \in E_z$, we use the natural isomorphisms
\eqref{eqn:globalTriv}, \eqref{eqn:bndlIso} and
$$
(\Lambda^{0,1}T^*\Sigma)\conjug \to \Lambda^{1,0}T^*\Sigma :
\overline{X^{0,1}} \mapsto X^{1,0}
$$
to obtain
\begin{equation*}
\begin{split}
\langle \beta \bar{\lambda} , \eta \rangle_E &=
\langle \beta (X^{1,0} \otimes \bar{\xi}) , \eta \rangle_E =
\langle X^{1,0} \otimes \beta\bar{\xi} , \eta \rangle_E 
= \langle X^{1,0} \otimes X^{0,1} \otimes \zeta , \eta \rangle_E \\
&= \left\langle \langle X,X \rangle_\Sigma \frac{1}{|X|_\Sigma^2} B\xi(X) , \eta \right\rangle_E
= \langle B\xi(X) , \eta \rangle_E.
\end{split}
\end{equation*}
\end{proof}

\begin{proof}[Proof of Proposition~\ref{prop:Weitzenbock}]
Writing $\mathbf{D}_B^* = \mathbf{D}^* + B^*$, we first expand
$$
\mathbf{D}_B^* \mathbf{D}_B \eta = (\mathbf{D}^* + B^*)(\mathbf{D} + B) \eta
= \mathbf{D}^*\mathbf{D} \eta + B^*B \eta + \mathbf{D}^*(B\eta) +
B^*(\mathbf{D} \eta).
$$
We will see that all derivatives of $\eta$ cancel in the sum of the last two
terms.  Write $B\eta = \beta \bar{\eta}$, where $\beta \in 
\Gamma(\Hom_\CC(E\conjug,F))$.
To understand $\mathbf{D}^*(B\eta) = \mathbf{D}^*(\beta\bar{\eta})$,
we can view $-\mathbf{D}^*$ as an anti-Cauchy-Riemann type operator on $F$,
and since $\mathbf{D}\conjug$ is likewise an anti-Cauchy-Riemann type
operator on $E\conjug$, Lemma~\ref{lemma:induced} provides an anti-Cauchy-Riemann
type operator $\p_H$ on $\Hom_\CC(E\conjug,F)$ such that
\begin{equation}
\label{eqn:applyLeibniz}
- \mathbf{D}^*(\beta\bar{\eta}) = (\p_H \beta) \bar{\eta} + 
\beta \, \mathbf{D}\conjug \bar{\eta}.
\end{equation}
For the final term in the expansion, observe that for any
$z \in \Sigma$, $\xi \in E_z$ and $\lambda \in F_z$,
$$
\Re \langle \lambda , B\eta \rangle_F = \Re \langle \lambda , \beta\bar{\eta} \rangle_F
= \Re \langle \beta^\dagger \lambda , \bar{\eta} \rangle_{E\conjug}
= \Re \langle \eta , \bar{\beta}^\dagger \bar{\lambda} \rangle_E
= \Re \langle \bar{\beta}^\dagger \bar{\lambda} , \eta \rangle_E,
$$
which gives the formula $B^* \lambda = \bar{\beta}^\dagger \bar{\lambda}$,
hence
\begin{equation}
\label{eqn:applyBstar}
B^*(\mathbf{D} \eta) = \bar{\beta}^\dagger \, \mathbf{D}\conjug \bar{\eta}.
\end{equation}
Putting \eqref{eqn:applyLeibniz} and \eqref{eqn:applyBstar} together and
applying Lemma~\ref{lemma:symmetry}, we have
$$
\mathbf{D}^*(B\eta) + B^*(\mathbf{D} \eta) =
- (\p_H \beta) \bar{\eta} + (\bar{\beta}^\dagger - \beta) \, \mathbf{D}\conjug \bar{\eta}
= - (\p_H \beta) \bar{\eta},
$$
and the stated formula follows by using the natural identification of 
$\Hom_\CC(E\conjug,F)$ with $\overline{\Hom}_\CC(E,F)$ to view $\p_H$ as an
anti-Cauchy-Riemann type operator on the latter.
\end{proof}

Suppose next that $(\widetilde{\Sigma},\tilde{\jmath})$ is another closed
connected Riemann surface.

\begin{defn}
\label{defn:pullbackCR}
Given a nonconstant holomorphic map 
$\varphi : (\widetilde{\Sigma},\tilde{\jmath}) \to (\Sigma,j)$
and a Cauchy-Riemann type operator $\mathbf{D}$ on $E$,
define $\varphi^*\mathbf{D}$ to be the unique Cauchy-Riemann type operator 
on $\varphi^*E$ that satisfies
\begin{equation}
\label{eqn:coveringOperator}
(\varphi^*\mathbf{D})(\eta \circ \varphi) = \varphi^*(\mathbf{D}\eta)
\quad\text{ for all } \quad \eta \in \Gamma(E).
\end{equation}
\end{defn}
The uniqueness of $\varphi^*\mathbf{D}$ is clear from 
\eqref{eqn:coveringOperator}.  To see that such an operator always exists,
write $\mathbf{D} = \mathbf{D}^\CC + A$ where $\mathbf{D}^\CC$ is
a complex-linear Cauchy-Riemann type operator and $A : E \to F$ is a
real-linear bundle map, which we can view equivalently as a $(0,1)$-form valued
in $\End_\RR(E)$.  Then $\mathbf{D}^\CC$ induces a holomorphic bundle structure
on $E$, which pulls back to define a holomorphic structure on $\varphi^*E$
and consequently a Cauchy-Riemann type operator $\varphi^*\mathbf{D}^\CC$.
The operator $\varphi^*\mathbf{D}^\CC + \varphi^*A$ then satisfies
\eqref{eqn:coveringOperator}.

\begin{example}
If $u : (\Sigma,j) \to (M,J)$ is an immersed $J$-holomorphic curve and
$\tilde{u} = u \circ \varphi$, then $\mathbf{D}_{\tilde{u}}^N =
\varphi^*\mathbf{D}_u^N$.
\end{example}

The next lemma is only interesting when $\varphi$ has branch points and is
thus not needed for the proof of Theorem~\ref{thm:super}, but the general
case of Proposition~\ref{prop:linear} requires it.  Given $\mathbf{D}$
and $B$ as in Proposition~\ref{prop:Weitzenbock} and a nonconstant
holomorphic map $\varphi : (\widetilde{\Sigma},\tilde{\jmath})
\to (\Sigma,j)$, let us abbreviate
$$
\widetilde{E} = \varphi^*E, \qquad \widetilde{F} = \Lambda^{0,1}T^*\widetilde{\Sigma} \otimes \widetilde{E},
\qquad
\widetilde{\mathbf{D}} = \varphi^*\mathbf{D} : 
\Gamma(\widetilde{E}) \to \Gamma(\widetilde{F}).
$$
Viewing $B$ as an $\overline{\End}_\CC(E)$-valued $(0,1)$-form
on~$\Sigma$, we can then define
$$
\widetilde{B} = \varphi^*B \in \Omega^{0,1}(\widetilde{\Sigma},
\overline{\End}_\CC(\widetilde{E})), \qquad
\widetilde{\mathbf{D}}_B = \widetilde{\mathbf{D}} + \widetilde{B} :
\Gamma(\widetilde{E}) \to \Gamma(\widetilde{F}).
$$
Choose a Hermitian structure $\langle\ ,\ \rangle_{\widetilde{\Sigma}}$
on $T\widetilde{\Sigma}$, whose real part is then a
$\tilde{\jmath}$-invariant Riemannian metric on $\widetilde{\Sigma}$.
The bundles $\widetilde{E}$
and $\widetilde{F}$ now inherit natural Hermitian structures, the former
as the pullback of $E$ and the latter as the tensor product
$\Lambda^{0,1}T^*\widetilde{\Sigma} \otimes \widetilde{E}$, and these
determine formal adjoint
operators $\widetilde{\mathbf{D}}^*$ and $\widetilde{\mathbf{D}}_{B}^*$.
The symmetry assumption \eqref{eqn:symmetryCondition} on $B$ implies 
that $\widetilde{B}$ also
satisfies this condition, so that Proposition~\ref{prop:Weitzenbock} gives
a Weitzenb\"ock formula over $\widetilde{\Sigma}$ in the form
$$
\widetilde{\mathbf{D}}_{B}^* \widetilde{\mathbf{D}}_{B} \eta
= \widetilde{\mathbf{D}}^* \widetilde{\mathbf{D}} \eta +
\widetilde{B}^*\widetilde{B} \eta - (\tilde{\p}_H \widetilde{B}) \eta
$$
for some anti-Cauchy-Riemann type operator $\tilde{\p}_H$ on
$\overline{\Hom}_\CC(\widetilde{E},\widetilde{F})$.

\begin{lemma}
\label{lemma:branched}
Assume the Riemannian metric $\Re \langle\ ,\ \rangle_{\widetilde{\Sigma}}$
on $\widetilde{\Sigma}$ is flat near all
critical points of~$\varphi$.  Then there exists a constant $c > 0$ such that 
$$
\big| \tilde{\p}_H \widetilde{B}(z) \big| \le
c | d\varphi(z) |^2
\quad\text{ for all }\quad
z \in \widetilde{\Sigma}.
$$ 
\end{lemma}
\begin{proof}
Recall from the proof of Proposition~\ref{prop:Weitzenbock} that after
identifying $\overline{\Hom}_\CC(\widetilde{E},\widetilde{F})$ with
$\Hom_\CC(\widetilde{E}\conjug,\widetilde{F})$ by writing
$\widetilde{B}\eta = \tilde{\beta} \bar{\eta}$
for $\tilde{\beta} \in \Gamma(\Hom_\CC(\widetilde{E}\conjug,\widetilde{F}))$,
the operator $\tilde{\p}_H$
is determined by the two anti-Cauchy-Riemann type operators
$\widetilde{\mathbf{D}}\conjug$ and $-\widetilde{\mathbf{D}}^*$ via a
Leibniz rule.  It will suffice to check that $|\tilde{\p}_H \tilde{\beta}|
\le c |d\varphi|^2$ holds in suitable local trivializations in a
neighborhood of each branch point $z_0 \in \widetilde{\Sigma}$.
Since the metric on $\widetilde{\Sigma}$ is assumed flat near $z_0$ 
and induces the same conformal structure as $\tilde{\jmath}$, 
we can find holomorphic coordinates $z = s + it$ on some neighorhood
$\widetilde{\uU} \subset \widetilde{\Sigma}$ of $z_0$ in which
the area form determined by the metric is $ds \wedge dt$,
and the induced bundle metric on $\Lambda^{0,1}T^*\widetilde{\Sigma}|_{\widetilde{\uU}}$ 
satisfies $| d\bar{z} |_{\widetilde{\Sigma}} = 1$.  Choose holomorphic
coordinates also on a neighborhood $\uU \subset \Sigma$
of $\varphi(z_0)$ and assume without loss of generality that
$\varphi(\widetilde{\uU}) = \uU$.  Next, fix a unitary trivialization of $E|_\uU$,
pull it back to define a trivialization of $\widetilde{E}|_{\widetilde{\uU}}$,
and use this together with the frame $d\bar{z}$ to trivialize
$\widetilde{F} = \Lambda^{0,1}T^*\widetilde{\Sigma} \otimes \widetilde{E}$
over~$\widetilde{\uU}$.  These trivializations identify $\mathbf{D}$ and 
$\widetilde{\mathbf{D}}$ locally with operators of the form
$$
\mathbf{D} = \dbar + A, \qquad \widetilde{\mathbf{D}} = \dbar + \tilde{A},
$$
where $\dbar = \p_s + i\p_t$, $A : \uU \to \End_\RR(\CC^m)$ and
$\tilde{A} : \widetilde{\uU} \to \End_\RR(\CC^m)$.  
Using the natural trivialization induced on $\widetilde{E}\conjug|_{\widetilde{\uU}}$ 
for which the canonical antilinear isomorphism $\widetilde{E} \to
\widetilde{E}\conjug$ appears as complex conjugation, $\widetilde{\mathbf{D}}\conjug$ can now
be written as
$$
\widetilde{\mathbf{D}}\conjug = \p + \tilde{A}\conjug,
$$
where $\tilde{A}\conjug : \widetilde{\uU} \to \End_\RR(\CC^m)$ is defined by
$\tilde{A}\conjug \bar{\eta} = \overline{\tilde{A} \eta}$.
Observe now that our
trivializations of $\widetilde{E}$ and $\widetilde{F}$ over $\widetilde{\uU}$
are both unitary, and since the area form $\widetilde{\uU}$ is also standard
in coordinates, the formal adjoint of $\widetilde{\mathbf{D}}$ takes the form
$$
\widetilde{\mathbf{D}}^* = -\p + \tilde{A}^\transpose.
$$
From these expressions and the Leibniz rule (cf.~the proof of
Lemma~\ref{lemma:induced}), one derives a function
$\widetilde{C} : \widetilde{\uU} \to \End_\RR(\End_\CC(\CC^m))$ such that the local
formula for $\tilde{\p}_H$ as a differential operator on
$\End_\CC(\CC^m)$-valued functions is
\begin{equation}
\label{eqn:pH}
\tilde{\p}_H = \p + \widetilde{C} \quad \text{ where } \quad
(\widetilde{C}\Phi) \bar{\eta} = -\tilde{A}^\transpose (\Phi\bar{\eta}) -
\Phi (\tilde{A}\conjug \bar{\eta}).
\end{equation}
Recall now that since $\widetilde{\mathbf{D}} = \varphi^*\mathbf{D}$,
$A$ and $\tilde{A}$ represent elements of
$\Omega^{0,1}(\Sigma,\End_\RR(E))$ and $\Omega^{0,1}(\widetilde{\Sigma},\End_\RR(\widetilde{E}))$
respectively, with the latter being the pullback of the former via~$\varphi$.
To make this explicit, the function $A : \uU \to \End_\RR(\CC^m)$
represents a $(0,1)$-form that corresponds under our 
trivialization of $E|_\uU$ to $d\bar{z} \otimes A \in \Omega^{0,1}(\End_\RR(\CC^m))$,
and $\tilde{A}$ then corresponds to the pullback
$\varphi^*(d\bar{z} \otimes A) = d\bar{\varphi} \otimes (A \circ \varphi)
= d\bar{z} \otimes \overline{\varphi'} \cdot (A \circ \varphi)$, giving the relation
$$
\tilde{A}(z) = \overline{\varphi'(z)} A(\varphi(z)).
$$
This implies an estimate of the form $|\tilde{A}(z)| \le c |\varphi'(z)|$ and,
by \eqref{eqn:pH}, a similar estimate for $|\widetilde{C}(z)|$.  Finally, 
viewing $\tilde{\beta}$ as a $(0,1)$-form valued in $\Hom_\CC(\widetilde{E}\conjug,\widetilde{E})$,
it is also the pullback of a $\Hom_\CC(E\conjug,E)$-valued $(0,1)$-form and is
thus similarly represented in trivializations by a function
$\tilde{\beta} : \widetilde{\uU} \to \End_\CC(\CC^m)$ that satisfies
$$
\tilde{\beta}(z) = \overline{\varphi'(z)} \beta(\varphi(z))
$$
for some function $\beta : \uU \to \End_\CC(\CC^m)$.  The estimate
$| \tilde{\p}_H \tilde{\beta} | = 
| \p \tilde{\beta} + \widetilde{C} \tilde{\beta} | \le c | \varphi' |^2$
now follows by a short calculation: indeed, $|\widetilde{C} \tilde{\beta} |
\le |\widetilde{C}| \cdot |\tilde{\beta}| \le c |\varphi'|^2$ for some
$c > 0$, and since $\overline{\varphi'}$ is antiholomorphic,
$\p\tilde{\beta} = \p\left( \overline{\varphi'} \cdot (\beta \circ \varphi) \right)
= \overline{\varphi'} (\p\beta \circ \varphi) \varphi'$
similarly satisfies $| \p\tilde{\beta} | \le c |\varphi'|^2$.
\end{proof}

}

\section{Regularity for the linearized operator}
\label{sec:linear}

We now state and prove a linear perturbation result that implies 
Proposition~\ref{prop:linear}.  {The result is a higher-dimensional
generalization of results for complex line bundles that were proved by
Taubes \cites{Taubes:counting,Taubes:SWtoGr}, and similar results
stated in \cite{Rauch}.}

Assume $(\Sigma,j)$ \rev{and $(\widetilde{\Sigma},\tilde{\jmath})$ are
closed connected Riemann surfaces, $\varphi : (\widetilde{\Sigma},\tilde{\jmath})
\to (\Sigma,j)$ is a holomorphic map of degree $d \ge 1$,}
$(E,J) \to (\Sigma,j)$ is a complex vector bundle of rank $m \ge 1$, 
and $\mathbf{D} : \Gamma(E) \to \Omega^{0,1}(\Sigma,E)$
is a real-linear Cauchy-Riemann type operator.  \rev{As in the previous
section, we shall abbreviate
$$
\widetilde{E} = \varphi^*E, \qquad \widetilde{\mathbf{D}} = \varphi^*\mathbf{D},
$$
where $\varphi^*\mathbf{D} : \Gamma(\varphi^*E) \to \Omega^{0,1}(\widetilde{\Sigma},\varphi^*E)$ 
denotes the induced Cauchy-Riemann type
operator on the pullback (see Definition~\ref{defn:pullbackCR}).}

Now assume $\ind(\mathbf{D}) = 0$.  By the Riemann-Roch formula, this
means
$$
- c_1(E) = m \chi(\Sigma) + c_1(E) = c_1(\overline{\Hom}_\CC(T\Sigma,E)),
$$
so there exists a complex-antilinear bundle isomorphism
$$
B : E \to \overline{\Hom}_\CC(T\Sigma,E).
$$
Choosing a Hermitian bundle metric \rev{$\langle\ ,\ \rangle_E$ on $E$},
we can also arrange
by Proposition~\ref{prop:symmetry} that $B$ satisfies the symmetry condition
\begin{equation}
\label{eqn:symm}
\rev{\Re \langle \xi , B\eta(X) \rangle_E = \Re \langle B\xi(X) , \eta \rangle_E}
\quad\text{ for all $(X,\xi,\eta) \in T\Sigma \oplus E \oplus E$}.
\end{equation}
This gives rise to a $1$-parameter family of real-linear Cauchy-Riemann type
operators on $\widetilde{E}$, defined by
$$
\rev{\widetilde{\mathbf{D}}_\tau = \varphi^*(\mathbf{D} + \tau B)
= \widetilde{\mathbf{D}} + \tau \widetilde{B}}
$$
for $\tau \in \RR$\rev{, where we abbreviate $\widetilde{B} := \varphi^*B$ 
with $B$ regarded as an $\overline{\End}_\CC(E,J)$-valued
$(0,1)$-form}.  Let $Z(d\varphi) \ge 0$ denote the algebraic count of
branch points of~$\varphi$, which is $-\chi(\widetilde{\Sigma}) + 
d \chi(\Sigma)$ by the Riemann-Hurwitz formula.  Then
\begin{equation*}
\begin{split}
\ind(\widetilde{\mathbf{D}}_\tau) &= m\chi(\widetilde{\Sigma}) + 2 c_1(\varphi^*E)
= m \left[ d\chi(\Sigma) - Z(d\varphi) \right] + 2 d c_1(E) \\
&= d \cdot \ind(\mathbf{D}) - m Z(d\varphi) = - m Z(d\varphi) \le 0.
\end{split}
\end{equation*}

\begin{thm}
\label{thm:linear2}
The operators $\widetilde{\mathbf{D}}_\tau : 
\Gamma(\widetilde{E}) \to \Omega^{0,1}(\Sigma,\widetilde{E})$
defined above are injective for all $\tau \in \RR$ outside of a discrete
subset.
\end{thm}

\begin{remark}
\label{remark:branched}
The proof of Theorem~\ref{thm:super} only requires the special case of
Theorem~\ref{thm:linear2} for which $\varphi : (\widetilde{\Sigma},\tilde{\jmath})
\to (\Sigma,j)$ is unbranched, and in this case \rev{the proof below becomes
somewhat simpler, e.g.~it does not require Lemma~\ref{lemma:branched}.}
The general case of
Theorem~\ref{thm:linear2} may nonetheless be useful for proving stronger
super-rigidity results.
\end{remark}

As in \S\ref{sec:Taubes}, we can use analytic perturbation theory to
reduce this theorem to a statement for particular values of~$\tau$.
We first extend $\widetilde{\mathbf{D}}_\tau$ to a Fredholm operator between
Hilbert spaces $H^1$ and $L^2$, each regarded as \emph{real} vector spaces
(since $\widetilde{\mathbf{D}}_\tau$ itself is real and not complex linear),
then complexify and consider the family of complex-linear Fredholm operators
$$
\widetilde{\mathbf{D}}_\tau : H^1(\widetilde{E}) \otimes \CC \to
L^2(\overline{\Hom}_\CC(T\widetilde{\Sigma},\widetilde{E})) \otimes \CC
$$
for $\tau \in \CC$.  This family depends holomorphically on $\tau$. 
Note that for $\tau\in\RR$, the underlying operator $\widetilde{\mathbf{D}}_\tau$ 
is injective whenever its complexification is injective. Thus
by Proposition~\ref{prop:analytic} in the appendix, 
in order to prove Theorem~\ref{thm:linear2}, 
it suffices to establish the following:

\begin{lemma}
\label{lemma:linear2}
The operator $\widetilde{\mathbf{D}}_\tau$ is injective for all sufficiently 
large $\tau > 0$.
\end{lemma}
\begin{proof}
\rev{Choose a Hermitian bundle metric on $T\widetilde{\Sigma}$ that matches
the standard Hermitian inner product in some choice of local holomorphic
coordinates near each of the branch points of~$\varphi$.  This gives
rise to a family of formal adjoint operators $\widetilde{\mathbf{D}}_\tau^*$
with $\widetilde{\mathbf{D}}_0^* =: \widetilde{\mathbf{D}}^*$ such that
by Proposition~\ref{prop:Weitzenbock},
$$
\widetilde{\mathbf{D}}_\tau^* \widetilde{\mathbf{D}}_\tau \eta
= \widetilde{\mathbf{D}}^* \widetilde{\mathbf{D}} \eta +
\tau^2 \widetilde{B}^*\widetilde{B} \eta - \tau (\tilde{\p}_H \widetilde{B}) \eta,
$$
and Lemma~\ref{lemma:branched} also implies
$$
\big| \tilde{\p}_H \widetilde{B} \big| \le c_1 |d\varphi|^2
$$
for some $c_1 > 0$.  Since $B$ is a bundle isomorphism, we can find another
constant $c_2 > 0$, such that $| B \eta | \ge c_2 |\eta|$ and thus
$$
\big| \widetilde{B} \eta \big| \ge c_2 |d\varphi| \cdot |\eta|.
$$
We then find for every $\eta \in \Gamma(\widetilde{E})$,
\begin{equation*}
\begin{split}
\| \widetilde{\mathbf{D}}_\tau \eta \|_{L^2}^2 &=
\left\langle \eta , \widetilde{\mathbf{D}}_\tau^* \widetilde{\mathbf{D}}_\tau \eta \right\rangle_{L^2}
= \left\langle \eta , \widetilde{\mathbf{D}}^* \widetilde{\mathbf{D}} \eta +
\tau^2 \widetilde{B}^*\widetilde{B} \eta - \tau (\tilde{\p}_H \widetilde{B}) \eta \right\rangle_{L^2} \\
&= \| \widetilde{\mathbf{D}} \eta \|_{L^2}^2 + \tau^2 \| \widetilde{B} \eta \|_{L^2}^2
- \tau \left\langle \eta , (\tilde{\p}_H \widetilde{B}) \eta \right\rangle_{L^2} 
\ge \left( \tau^2 c_2^2 - \tau c_1 \right) \big\| |d\varphi| \cdot \eta \big\|_{L^2}^2,
\end{split}
\end{equation*}
where the constants $c_1, c_2 > 0$ are independent of~$\eta$.
Since $|d\varphi| > 0$ almost everywhere, we conclude that
$\widetilde{\mathbf{D}}_\tau$ is injective whenever 
$\tau^2 c_2^2 - \tau c_1 > 0$.}
\end{proof}

\appendix

\section{Some analytic perturbation theory}
\label{sec:analytic}

The linear perturbation argument of \S\ref{sec:linear} requires a basic
ingredient from analytic perturbation theory in the spirit of
\cite{Kato}.  Since we were not able to find a reference for the precise
result we need, we have included a proof of it in this appendix for the sake of 
completeness.

Given complex Banach spaces $X$ and $Y$, denote by $\lL(X,Y)$ the Banach
space of bounded complex-linear operators $X \to Y$, 
abbreviate $\lL(X) := \lL(X,X)$, and let
$\Fred(X,Y) \subset \lL(X,Y)$ denote the open subset consisting of
Fredholm operators.  Since $\Fred(X,Y)$ carries a natural complex structure
as a subset of $\lL(X,Y)$, it makes sense to speak of holomorphic maps 
into $\Fred(X,Y)$, i.e.~maps which are Fr\'{e}chet differentiable
with complex-linear derivative.

\begin{prop}
\label{prop:analytic}
Suppose $\uU \subset \CC$ is a connected open subset and $\uU \to \Fred(X,Y) :
\tau \mapsto \mathbf{T}_\tau$ is a holomorphic map, and let
$$
Z = \{ \tau \in \uU\ |\ \text{$\mathbf{T}_\tau$ is not injective} \}.
$$
Then either $Z$ is a discrete subset of $\uU$, or $Z = \uU$.
\end{prop}
\begin{proof}
Given any $\mathbf{T}_0 \in \Fred(X,Y)$, there exist splittings into closed
linear subspaces
$$
X = V \oplus \ker\mathbf{T}_0, \qquad  Y = W \oplus \coker \mathbf{T}_0
$$
such that $\mathbf{T}_0|_{V}$ is an isomorphism $V \to W$.  Using this
splitting,
we can write any other $\mathbf{T} \in \Fred(X,Y)$ in block form as
$$
\mathbf{T} = \begin{pmatrix}
\mathbf{A} & \mathbf{B} \\
\mathbf{C} & \mathbf{D}
\end{pmatrix},
$$
and define $\oO \subset \Fred(X,Y)$ to be the
open neighborhood
of $\mathbf{T}_0$ for which the block $\mathbf{A}$ is invertible.
We can then define a holomorphic map
$$
\Phi : \oO \to \lL(\ker\mathbf{T}_0,\coker\mathbf{T}_0) :
\mathbf{T} \mapsto \mathbf{D} - \mathbf{C} \mathbf{A}^{-1} \mathbf{B}.
$$
We claim that for all $\mathbf{T} \in \oO$, $\ker \mathbf{T} \cong
\ker \Phi(\mathbf{T})$.  To see this, associate to $\mathbf{T}$ 
the isomorphism
$$
\Psi = \begin{pmatrix}
\1 & - \mathbf{A}^{-1} \mathbf{B} \\
0  & \1
\end{pmatrix} \in \lL(V \oplus \ker \mathbf{T}_0) = \lL(X).
$$
Then $\mathbf{T} \Psi = \begin{pmatrix}
\mathbf{A} & 0 \\
\mathbf{C} & \Phi(\mathbf{T})
\end{pmatrix}$,
and since $\mathbf{A}$ is invertible, 
$\ker\mathbf{T}\Psi = \{0\} \oplus \ker \Phi(\mathbf{T})$,
from which the claim follows.

Now if $\uU \to \Fred(X,Y) : \tau \to \mathbf{T}_\tau$ is a family of
operators depending holomorphically on~$\tau$, then fixing any
$\tau_0 \in \uU$ and placing $\mathbf{T}_{\tau_0}$ in the role of
$\mathbf{T}_0$ above, one can define $\Phi$ on a neighborhood of
$\mathbf{T}_{\tau_0}$ so that
$$
\tau \mapsto \Phi(\mathbf{T}_\tau)
$$
defines a holomorphic curve mapping into the finite-dimensional complex vector
space $\lL(\ker\mathbf{T}_{\tau_0},\coker\mathbf{T}_{\tau_0})$
for $\tau$ in a neighborhood of~$\tau_0$.  The set of all $\tau$
near~$\tau_0$ for which $\mathbf{T}_{\tau}$ is not injective then
corresponds to the intersections of this holomorphic curve with
the stratified complex subvariety of noninjective maps in
$\lL(\ker\mathbf{T}_{\tau_0},\coker\mathbf{T}_{\tau_0})$, which has
positive codimension.  The proposition thus follows from the standard
results on intersections of holomorphic curves with complex
submanifolds.
\end{proof}

\end{document}